\documentclass[12pt,oneside]{amsart}

\usepackage[margin=1in]{geometry}
\usepackage[section]{placeins}

\usepackage{mathabx}

\usepackage{graphicx}
\usepackage{tikz}
\usepackage{latexsym}
\usepackage{amsmath,amsthm,amscd,amssymb}
\usepackage{enumitem,multicol}
\usepackage{wasysym}
\usepackage{ulem}
\usepackage{soul}
\usepackage{multirow}
\usepackage{rotating}
\usepackage{pdflscape}
\usepackage{float}
\usepackage{dsfont}
\usepackage{comment}

\usepackage[colorlinks,citecolor=red,pagebackref,hypertexnames=false]{hyperref}
\hypersetup{%
  colorlinks = true,
  citecolor=red,
  linkcolor  = red
}

\setlist{itemsep=.06125in}

\restylefloat{table}

\numberwithin{equation}{section}

\theoremstyle{plain}
\newtheorem{theorem}{Theorem}[section]
\newtheorem{lemma}[theorem]{Lemma}

\theoremstyle{definition}
\newtheorem{definition}[theorem]{Definition}

\theoremstyle{remark}
\newtheorem{remark}[theorem]{Remark}

\newtheorem{case[theorem]}{Case}

\def\blue{\textcolor{blue}}
\def\red{\textcolor{red}}

\date{\today}      

\author{K. Aldaleh, W. Burstein, G. Garza, G. Hart, A. Iosevich, J. Iosevich, 
A. Khalil, J. King, N. Kulkarni, T. Le, I. Li, A. Mayeli, B. McDonald, 
K. Nguyen, and N. Shaffer}

\address{Department of Mathematics, University of Delaware, Newark DE, USA}
\email{giogarza@udel.edu}

\address{Department of Mathematics, University of Rochester, Rochester, NY, USA}
\email{iosevich@gmail.com}

\address{Department of Mathematics, CUNY Graduate Center, New York, NY, USA}
\email{amayeli@gc.cuny.edu } 

\thanks{A.~I. was supported in part by the National Science Foundation under NSF DMS-2154232. 
A.~M. was supported in part by the AMS-Simons Research Enhancement Grant and 
by the National Science Foundation under NSF DMS-2453769.} 

\title{The Fourier Ratio and complexity of signals}

\begin{document}

\begin{abstract}
We study the Fourier ratio of a signal $f:\mathbb Z_N\to\mathbb C$,
\[
\mathrm{FR}(f)\ :=\ \sqrt{N}\,\frac{\|\widehat f\|_{L^1(\mu)}}{\|\widehat f\|_{L^2(\mu)}}
\ =\ \frac{\|\widehat f\|_1}{\|\widehat f\|_2},
\]
as a simple scalar parameter governing Fourier-side complexity, structure, and learnability. Using the Bourgain--Talagrand theory of random subsets of orthonormal systems, we show that signals concentrated on generic sparse sets necessarily have large Fourier ratio, while small $\mathrm{FR}(f)$ forces $f$ to be well-approximated in both $L^2$ and $L^\infty$ by low-degree trigonometric polynomials. Quantitatively, the class $\{f:\mathrm{FR}(f)\le r\}$ admits degree $O(r^2)$ $L^2$-approximants, which we use to prove that small Fourier ratio implies small algorithmic rate--distortion, a stable refinement of Kolmogorov complexity.

We further show that $\mathrm{FR}(f)$ is the natural controlling parameter in a sharpened discrete Fourier uncertainty principle, relate $L^1$ and $L^2$ Fourier concentration, and prove that $\mathrm{FR}(f)$ is stable under random restriction and additive Gaussian noise. Motivated by examples on $\mathbb Z_{pq}$, we introduce the bi-Fourier ratio
\[
\mathrm{FR}_{\mathrm{bi}}(f):=\min\{\mathrm{FR}(f),\mathrm{FR}(\widehat f)\},
\]
which is invariant up to constants under the Fourier transform and better captures algorithmic complexity in settings where moving between $f$ and $\widehat f$ is computationally cheap.

On the learning-theoretic side, we bound the Vapnik–Chervonenkis and statistical query dimensions of the concept class $\{f:\mathrm{FR}(f)\le r\}$ via the low-degree approximation theorems, and we show that signals with small Fourier ratio admit accurate imputation of missing values from a number of random samples comparable to their effective Fourier sparsity, using tools from compressed sensing. Finally, we prove a Chang-type lemma for general signals, showing that small Fourier ratio forces strong additive structure in the set of large Fourier coefficients, and we compute the Fourier ratio for several real-world time series while numerically probing the constants in the underlying Talagrand- and Bourgain-type inequalities.
\end{abstract}

\maketitle

\tableofcontents

\section{Introduction}
\label{section:introduction} 

Let $f: {\mathbb Z}_N \to {\mathbb C}$ be a signal. The purpose of this paper is to explore a simple mechanism to determine whether $f$ contains sufficient structure to learn this signal using a relatively small number of samples. A simple, common-sense idea is that it is challenging to learn a ``random signal", but the task becomes considerably easier in the presence of structure. The question that has been addressed by many authors in a variety of contexts is how to determine the degree to which a given signal behaves like a random one. The central theme of this paper is that the quantity 
\begin{equation} \label{eq:mama} 
    \text{FR}(f) \equiv \sqrt{N} \cdot \frac{{\|\widehat{f}\|}_{L^1(\mu)}}{{\|\widehat{f}\|}_{L^2(\mu)}}=\frac{{\|\widehat{f}\|}_{L^1({\mathbb Z}_N)}}{{\|\widehat{f}\|}_{L^2({\mathbb Z}_N)}}.
\end{equation}
is an effective detector of randomness and structure, where here and throughout, if $g: {\mathbb Z}_N \to {\mathbb C}$, 
\begin{equation} \label{eq:probnorm}
    {\|g\|}_{L^p(\mu)}={\left( \frac{1}{N} \sum_{x \in {\mathbb Z}_N} {|g(x)|}^p \right)}^{\frac{1}{p}},
\end{equation}
\begin{equation} \label{eq:lpnorm}
    {\|g\|}_{p}={\left( \sum_{x \in {\mathbb Z}_N} {|g(x)|}^p \right)}^{\frac{1}{p}},
\end{equation}
and 
$$ \widehat{f}(m)=\frac{1}{\sqrt{N}} \sum_{x \in {\mathbb Z}_N} \chi(-xm) f(x),$$ where $\chi(t)=e^{\frac{2 \pi i t}{N}}$. 

\vskip.125in 

The initial draft of this paper was written during the StemForAll2025 (\cite{stemforall2025}) research program organized by the 5th and 12th listed authors in July/August 2025. The paper was finished under the auspices of the continuous Vertical Integration of Research program, also supervised by the 5th and the 12th listed authors of the paper, designed to bring undergraduates, graduate students, postdocs, and faculty together to work on research.

\vskip.125in 

\subsection{Structure of this paper} \label{subsection:structure} We are now going to describe the structure of the paper. \begin{itemize} 

\item In Subsection \ref{subsection:notation}, we lay down the notation frequently used in this paper. 

\item In Subsection \ref{subsection:talagrand}, we state the celebrated result due to Talagrand (\cite{Talagrand98}) on random subsets of sets of orthogonal functions and use it to illustrate the Fourier Ratio $\text{FR}(f)$ (\ref{eq:mama}) in a natural context in Subsection \ref{subsection:firstappearance}. 

\item In Subsection \ref{subsection:aboveandbelow}, we study the range of the Fourier Ratio, and show that $1 \leq \text{FR}(f) \leq \sqrt{N}$. 

\item In Subsection \ref{subsection:FRcontrollingparameter}, we prove refined lower and upper bounds for the Fourier Ratio and show that it is a natural controlling parameter in the classical Fourier Uncertainty Principle. 

\item In Subsection \ref{subsection:keyexample} we study two key illustrative examples. The first is the indicator function of ${\mathbb Z}_p$ embedded in ${\mathbb Z}_{pq}$, where $p$ and $q$ are distinct primes. This example shows that $\text{FR}(f)$ and $\text{FR}(\widehat{f})$ should be considered in analyzing the complexity of a signal. The second is an inverse Fourier transform of a random set. In spite of the randomness, we prove in this paper that it is relatively easy to impute the missing values of such a signal from small samples. 

\item In Subsection \ref{subsection:main}, we prove that if the Fourier Ratio is small, then the signal can be well-approximated by a low degree trigonometric polynomial, either in $L^{\infty}$ or $L^2$ norms using probabilistic tools. 

\item In Subsection \ref{subsection:CKScomplexity}, we leverage the results from Subsection \ref{subsection:main} to prove that signals with small Fourier Ratio have small Kolmogorov complexity, or, more precisely, small algorithmic rate distortion (\ref{eq:algorithmicratedistortion}). 

\item In Subsection \ref{susbsection:imputation}, we use the theory of the Fourier Ratio built up to that point, along with some ideas from compressed sensing to show that a small Fourier Ratio allows one to impute the missing values in time series with high accuracy under some realistic assumptions on the size of the missing set. In the process, we show that the Fourier Ratio of a signal does not change much if we restrict it to a sufficiently large random subset. 

\item In Subsection \ref{subsection:dimensions}, we show that both the Vapnik-Chervonenkis and statistical query dimension of the class of signals with a small Fourier Ratio are small. Quantitative bounds are provided. 

\item In Subsection \ref{subsection:stability}, we prove a series of results showing that the Fourier Ratio is stable under small perturbations. 

\item In Subsection \ref{subsection:chang}, we study the Fourier Ratio in the context of Chang's lemma from additive number theory. We show that a small Fourier Ratio implies a considerable amount of additive structure. 

\item Finally, in Section \ref{section:numerical}, we perform a series of numerical experiments to illustrate some of the results in this paper. First, we compute the Fourier Ratio for several very different real-life data sets and show that the Fourier Ratio is quite small. We compute the Fourier Ratio of a random data set and show that the Fourier Ratio is very large, as the theory suggests. Our final set of numerical experiments explores the value of the constant in Talagrand's inequality, as it has a practical bearing on the applications established in this paper. 

\end{itemize} 

\vskip.125in 

\subsection{Normalization and notation} \label{subsection:notation} Throughout this paper, we work on $\mathbb{Z}_N$ with the unitary Fourier transform
$$ \widehat{f}(m)=\frac{1}{\sqrt{N}}\sum_{x\in\mathbb{Z}_N} e^{-\frac{2\pi i xm}{N}} f(x).$$ From this definition of the Fourier transform, one can show the Fourier inversion property, which states that
$$ f(x)=\frac{1}{\sqrt{N}}\sum_{m\in\mathbb{Z}_N} e^{\frac{2\pi i xm}{N}} \widehat{f}(m).$$ 
The Plancherel identity holds in both $\ell^2$ and probability-normalized $L^2(\mu)$:
\[
\|\widehat{f}\|_2=\|f\|_2,
\qquad
\|\widehat{f}\|_{L^2(\mu)}=\|f\|_{L^2(\mu)}.
\]
We frequently convert between $\ell^p$ and $L^p(\mu)$ using 
$$ \|g\|_{L^p(\mu)}=N^{-\frac{1}{p}}\|g\|_{p}
\quad(1\le p<\infty),\qquad
\|g\|_{L^\infty(\mu)}=\|g\|_\infty.$$

\vskip.25in 

\subsubsection{Fourier Ratio} With this normalization,
\[
\text{FR}(f)\ :=\ \sqrt{N}\,\frac{\|\widehat{f}\|_{L^1(\mu)}}{\|\widehat{f}\|_{L^2(\mu)}}
\ =\ \frac{\|\widehat{f}\|_1}{\|\widehat{f}\|_2}
\ \in[1,\sqrt{N}],
\]
is scale-invariant. More precisely, $\text{FR}(\alpha f)=\text{FR}(f)$ for all $\alpha\neq 0$. The extreme cases are 
$\text{FR}(f)=1$ iff $\widehat{f}$ is $1$-sparse, and $\text{FR}(f)=\sqrt{N}$ when $|\widehat{f}(m)|$ is constant.

\medskip
\subsubsection{Numerical sparsity and coherence} The numerical sparsity of a vector $x$ is 
$$\mathrm{ns}(x):=\big(\|x\|_1/\|x\|_2\big)^2,$$ hence
$\text{FR}(f)^2=\mathrm{ns}(\widehat{f})$. See, for example, \cite{BandeiraLewisMixon2015}. We also use the standard notion of coherence,
$$ \mu(f)\ :=\ \frac{N\|f\|_\infty^2}{\|f\|_2^2}\in[1,N],$$ 
which controls random-restriction sampling rates in our results. Similar quantities also come up in the theory of portfolio management (see e.g. \cite{Sharpe1994}). 

\vskip.25in 
\subsubsection{Bi-Fourier ratio} Motivated by the example in Subsection \ref{subsection:keyexample} below, we summarize Fourier complexity using the quantity 
$$ \text{FR}_\mathrm{bi}(f):=\min\{\text{FR}(f),\,\text{FR}(\widehat{f})\}.$$

This is invariant (up to constants) under applying the Fourier transform or its inverse.

\vskip.125in 

\subsection{Some key background results} Our starting point is the following result due to Talagrand (\cite{Talagrand98}) \label{subsection:talagrand}

\begin{theorem}\label{thm:Talagrand98}
Let $(\varphi_j)_{j=1}^n$ be an orthonormal system in 
$L^2(\mathbb{Z}_N)$ with $\|\varphi_j\|_{L^\infty} \leq K$ for $1 \leq j \leq n$.  
There exists a constant $\gamma_0 \in (0,1)$ and a subset 
$I \subset \{1, \dots, n\}$ with $|I| \ge \gamma_0 n$  
such that for every $a = (a_i) \in \mathbb{C}^n$,
$$
\left( \sum_{i \in I} |a_i|^2 \right)^{1/2} 
\leq C_T \, K \, \big( \log(n) \, \log \log(n) \big)^{1/2} 
\left\| \sum_{i \in I} a_i \varphi_i \right\|_{L^1},
$$
where $C_T > 0$ is a universal constant.
\end{theorem}

We shall need the following version, stated for signals on ${\mathbb Z}_N$, with the roles of $h$ and $\widehat{h}$ reversed. See \cite{BIMN2025}. 

\begin{definition} \label{def:generic} Let $0< p <1$. We say  a random set $S \subset [n]= \{0, 1, \cdots, n-1\}$ is generic if each element of $[n]$ is selected independently with probability $p$.
\end{definition} 

The following result can be deduced from Theorem \ref{thm:Talagrand98} (see \cite{BIMN2025}). 

\begin{theorem} \label{thm: combo} There exists $\gamma_0 \in (0,1)$ such that if $h: {\mathbb Z}_N \to {\mathbb C}$ supported in a generic set $M$ (in the sense of Definition \ref{def:generic}) of size $\gamma_0 \frac{N}{\log(N)}$, then with probability $1-o_N(1)$, 
\begin{equation}\label{eq: talagrand1} {\left( \frac{1}{N}  \sum_{m \in {\mathbb Z}_N} {|\widehat{h}(m)|}^2 \right)}^{\frac{1}{2}} \leq C_T {(\log(N) \log \log(N))}^{\frac{1}{2}} \cdot \frac{1}{N} \sum_{m \in {\mathbb Z}_N} |\widehat{h}(m)|, \end{equation} where $C_T > 0$ is a constant that depends only on $\gamma_0$. 
\end{theorem} 

\vskip.125in 

\begin{remark} \label{rmk:talagrandsharpness} It is known (see \cite{Talagrand98}) that $\sqrt{\log(N)}$ in (\ref{eq: talagrand1}) cannot, in general, be removed, and it is not known whether the removing the remaining $\sqrt{\log \log(N)}$ is possible. \end{remark} 

In some cases, the term ${(\log(N) \log \log(N))}^{\frac{1}{2}}$ can be removed. The following result, stated in the setting of this paper, is due to Bourgain (\cite{Bourgain89}). 

\begin{theorem} \label{thm:bourgain} Suppose that $M$ is generic, as above, $|M|= \lceil N^{\frac{2}{q}} \rceil$, $q>2$. Then for all $h: {\mathbb Z}_N \to {\mathbb C}$, supported in $M$, 
\begin{equation} \label{eq:bourgain} {\left( \frac{1}{N} \sum_{m \in {\mathbb Z}_N} {|\widehat{h}(m)|}^q \right)}^{\frac{1}{q}} \leq C(q) \cdot {\left( \frac{1}{N} \sum_{m \in {\mathbb Z}_N} {|\widehat{h}(m)|}^2 \right)}^{\frac{1}{2}} \end{equation} 

It is not difficult to deduce from the proof that if $M$ is chosen randomly, then the bound holds with probability $1-\varepsilon$ if $C(q)$ is replaced by $\frac{C(q)}{\varepsilon}$. 
\end{theorem} 

\vskip.125in 

We learned the following observation from William Hagerstrom (\cite{Hagerstrom2025}), which can be proven using H\"older's inequality. 

\begin{lemma} \label{lemma:vershynintrick} Suppose that for $h: {\mathbb Z}_N^d \to {\mathbb C}$, 
\begin{equation} \label{eq:bourgainproxy} {\left( \frac{1}{N} \sum_{x \in {\mathbb Z}_N^d} {|\widehat{h}(x)|}^q \right)}^{\frac{1}{q}} \leq C(q) {\left( \frac{1}{N} \sum_{x \in {\mathbb Z}_N} {|\widehat{h}(x)|}^2 \right)}^{\frac{1}{2}} \end{equation} for some $q>2$. 

Then 
$$ {\left( \frac{1}{N} \sum_{x \in {\mathbb Z}_N} {|\widehat{h}(x)|}^2 \right)}^{\frac{1}{2}} \leq {(C(q))}^{\frac{q}{q-2}} \cdot \frac{1}{N} \sum_{x \in {\mathbb Z}_N} |\widehat{h}(x)|. $$
\end{lemma} 

\begin{remark} \label{rmk:nolog} It follows that if the expected size of $M$ in Theorem \ref{thm: combo} is $O(N^{1-\varepsilon})$ for some $\varepsilon>0$, then 
$C_T \sqrt{\log(N) \log \log(N)}$ in (\ref{eq: talagrand1}) can be replaced by $C'_T$, independent of $N$. \end{remark} 

\vskip.125in 

\subsection{The Fourier ratio makes its first appearance} \label{subsection:firstappearance} We are now ready to exploit the setup above. Let $S \subset {\mathbb Z}_N$, chosen randomly. Then in view of Theorem \ref{thm: combo} and Remark \ref{rmk:nolog}, with very high probability, we have 
\begin{equation} \label{eq:mamaindicator} {\|\widecheck{1}_S\|}_{L^2(\mu)} \leq C_T {\|\widecheck{1}_S\|}_{L^1(\mu)}.\end{equation}  

Using Plancherel, we see that 
$$ {\|\widecheck{1}_S\|}_{L^1(\mu)} \ge \frac{1}{C_T} \sqrt{\frac{|S|}{N}},$$ or, in other words, 
$$ \frac{{\|\widecheck{1}_S\|}_{L^1(\mu)}}{{\|\widecheck{1}_S\|}_{L^2(\mu)}} \ge \frac{1}{C_T},$$ the largest possible size up to a constant. It is not difficult to see that the same conclusion holds if $\widecheck{1}_S$ is replaced by an arbitrary function with the Fourier transform supported in $S$.  

On the other hand, let's consider the case where $S$ is far from random. Suppose that $S={\mathbb Z}_N$. Then 
$$ \widecheck{1}_S(x)=\sqrt{N} \cdot \delta_0(x),$$ where $\delta_0(x)=1$ if $x=0$ and $0$ otherwise. it is not difficult to check that in this case, 
\begin{equation} \label{eq:worstcase} \frac{{\|\widecheck{1}_S\|}_{L^1(\mu)}}{{\|\widecheck{1}_S\|}_{L^2(\mu)}}=\frac{1}{\sqrt{N}}, \end{equation} a much smaller quantity than what we get in (\ref{eq:mamaindicator}) in the generic case. This suggests that the quantity (\ref{eq:mama}) in the case $f=\widecheck{1}_S$ may contain significant information indicating the degree to which a set $S$ is random. 

\vskip.125in 

In the remainder of this paper, we explore to what extent the Fourier Ratio captures the degree to which the signal $f$ is learnable, in a variety of different senses. We also apply the resulting ideas to the imputation of time series with missing values. 

\vskip.125in 

\subsection{Bounding the Fourier ratio from above and below.} \label{subsection:aboveandbelow}

Note that by Fourier inversion and the triangle inequality,
$$|f(x)| \leq N^{-\frac12} \sum_{m\in\mathbb Z_N} |\widehat f(m)| = N^{-\frac12} \|\widehat f\|_1.$$
Squaring both sides, summing over $x\in\mathbb Z_N$, and taking square roots yields
$$\|f\|_2 \leq \|\widehat f\|_1,$$
so by Plancherel, we see that
$$\text{FR}(f) = \frac{\|\widehat f\|_1}{\| f\|_2} \geq 1,$$
and we saw in the example described above (\ref{eq:worstcase}) that this lower bound is realized by the constant function $1$. Note also that by Cauchy-Schwarz, we have that
$$\text{FR}(f) = \frac{\|\widehat f\|_1}{\|\widehat f\|_2} \leq \frac{\sqrt N \|\widehat f\|_2}{\|\widehat f\|_2} = \sqrt N.$$
Thus we have just seen that
$$1 \leq \text{FR}(f) \leq \sqrt N.$$



\vskip.125in 

The same argument used to bound $\text{FR}(f)$ below shows that if $f$ is supported in $E \subset {\mathbb Z}_N$, then 
\begin{equation} \label{eq:epsilonsupportlowerbound} \text{FR}(f) \ge \frac{\sqrt{N}}{\sqrt{|E|}}. \end{equation} 

To see that (\ref{eq:epsilonsupportlowerbound}) can be realized, let $E=q{\mathbb Z}_p$, $p$ prime, be a copy of ${\mathbb Z}_p$ sitting inside ${\mathbb Z}_{pq}$ in the obvious way. Then, for the setting $N=pq$, we have 
$$ \widehat{1}_E(m)=\frac{1}{\sqrt{N}} \sum_{k=0}^{p-1} e^{-\frac{2 \pi i qkm}{N}}  = 
\begin{cases}
\frac{p}{\sqrt{N}} 1_S(m), & \text{if } p \mid m,\\[4pt]
0, & \text{if } p \nmid m.
\end{cases}
$$ where $S=p{\mathbb Z}_q$, sitting inside ${\mathbb Z}_{pq}$ in the natural way. Setting $f=1_E$, it follows that 
\begin{equation} \label{eq:frforzp} \text{FR}(f)=\sqrt{N} \cdot \frac{{\|\widehat{f}\|}_{L^1(\mu)}}{{\|\widehat{f}\|}_{L^2(\mu)}}=\frac{\sqrt{N}}{\sqrt{p}}=\frac{\sqrt{N}}{\sqrt{|E|}}. \end{equation} 

\vskip.25in 

\subsection{Refined lower and upper bounds on the Fourier Ratio and the classical Fourier Uncertainty Principle} \label{subsection:FRcontrollingparameter} 

We now refine both the lower and the upper bounds for the Fourier Ratio and demonstrate that the Fourier Ratio is a natural controlling parameter in the classical Fourier Uncertainty Principle. 

\begin{theorem} \label{theorem:FRuncertainty} Let $f: {\mathbb Z}_N \to {\mathbb C}$, $L^2$-concentrated in $E \subset {\mathbb Z}_N$ at level $a \in (0,1)$, in the sense that 
$$ {||f||}_{L^2(E^c)} \leq a {||f||}_{L^2({\mathbb Z}_N)},$$ with $\widehat{f}$
$L^1$-
concentrated
on $S \subset {\mathbb Z}_N$ at level $b\in (0,1)$, in the sense that 
$$ {||\widehat{f}||}_{L^1(S^c)} \leq b{||\widehat{f}||}_{L^1({\mathbb Z}_N)}.$$ 

Then 
\begin{equation}\label{eq:FRcontrollingparameter}  (1-a)^2  \cdot \frac{N}{|E|} \leq {\text{FR}(f)}^2 \leq \frac{|S|}{{(1-b)}^2}. \end{equation} 

\end{theorem} 

\vskip.125in 

In particular, 
\begin{equation} \label{eq:FRintermediary} |E| \cdot |S| \ge {(1-a)}^2 \cdot {(1-b)}^2 \cdot N, \end{equation} a version of the classical uncertainty principle (see e.g. \cite{DS89}). While (\ref{eq:FRintermediary}) is well-known, (\ref{eq:FRcontrollingparameter}) shows that the Fourier Ratio is a natural controlling parameter in the Fourier Uncertainty Principle.  

\subsubsection{$L^1$ versus $L^2$ concentration} It is important to note that the conclusion of Theorem $\ref{theorem:FRuncertainty}$ does not hold if we replace the hypothesis that $\widehat{f}$ is $L^1$-concentrated on $S$ at level $b$ with the hypothesis that $\widehat{f}$ is $L^2$-concentrated on $S$ at level $b$. This is shown by the following example: for any $S\subset \mathbb{Z}_N$, let
\[\widehat{f}(m) = \begin{cases} 
\sqrt{1-b^2}/|S|^{1/2}& m \in S \\
b/(N-|S|)^{1/2}&m \notin S
\end{cases}.\]
It is easy to verify that $\widehat{f}$ is $L^2$-concentrated in $S$ at a level $b$, since $||\widehat{f}||_{L^2(S^C)} = b$ and $||\widehat{f}||_{2} = 1$.
The Fourier ratio is
$$\mathrm{FR}(f) = b\cdot(N-|S|)^{1/2} + \sqrt{1-b^2}\cdot |S|^{1/2}.$$ 
We now claim that $\mathrm{FR}(f)^2 \geq |S|/(1-b)^2$ if and only if
\[|S| \leq \frac{b^2(1-b)^2}{(1-(1-b)\sqrt{1-b^2})^2} \cdot (N-|S|).\]
Indeed, $\mathrm{FR}(f) \geq |S|^{1/2}/(1-b)$ is equivalent to the statement
\[(1-b)\cdot b\cdot \sqrt{(N-|S|)} \geq (1-(1-b)\sqrt{1-b^2}) \cdot \sqrt{|S|}, \]
and the claim follows.

\vskip.125in 

\noindent We now give a result directly linking the $L^1$ and $L^2$ concentrations, illustrating that the relationship in the example above is extremal up to a logarithmic factor. 

\begin{theorem} \label{thm:relation-between-L1-L2-concentration}

Suppose that $f:\mathbb{Z}_N \rightarrow \mathbb{C}$ is an indicator function such that $\widehat{f}$ is $L^2$-concentrated on a set $S$ with accuracy $b$, i.e.
\[||\widehat{f}||_{L^2(S^C)} \leq b \cdot ||\widehat{f}||_2.\]
We will also assume that $S$ is minimal with respect to the $L^2$-concentration property in the following sense: for any set $T \subset S$, we assume that
\[||\widehat{f}||^2_{L^2(T)} < (1-b^2)\cdot ||\widehat{f}||^2_{2} \cdot \frac{|T|}{|S|}.\]
Then 
\[\frac{||\widehat{f}||_{L^1(S)}}{||\widehat{f}||_{L^2(S)}} \geq \frac{2}{(\log(N))^{3/2}}\cdot \frac{1}{\sqrt{1-b^2}}\cdot \frac{1}{\left(1+(\log N)^{3/2}\cdot\frac{b \cdot |S^C|^{1/2}}{\sqrt{1-b^2}\cdot |S|^{1/2}}\right)} \cdot \text{FR}(f).\]
Consequently (since the fact that $f$ is $L^2$-concentrated on $S$ tells us that $||\widehat{f}||_2 \leq \frac{1}{\sqrt{1-b^2}}\cdot ||\widehat{f}||_{L^2(S)}$), we see that
$\widehat{f}$ is $L^1$-concentrated on $S$ with the following accuracy:
\[||\widehat{f}||_{L^1(S)} \geq \frac{\sqrt{2}}{(\log(N))^{3/2}}\cdot\frac{1}{\left(1+(\log N)^{3/2}\cdot\frac{b \cdot |S^C|^{1/2}}{\sqrt{1-b^2}\cdot |S|^{1/2}}\right)} \cdot ||\widehat{f}||_1.\]
\end{theorem}

\subsection{Two key examples} \label{subsection:keyexample}
In this subsection we give two somewhat counter-intuitive examples that illustrate the subtleties of the Fourier ratio.

\subsubsection{Small subgroups of ${\mathbb Z}_{pq}$}
The example given above where $N=pq$, $p,q$ distinct primes, and $f=1_E$, where $E$ denotes ${\mathbb Z}_p$ embedded in ${\mathbb Z}_{pq}$ in a natural way, points to a very interesting phenomenon that we need to address.

As we noted above, $\text{FR}(f)$ in this case is equal to $\sqrt{\frac{N}{p}}$. If $p$ is much larger than $q$, then $\text{FR}(f)$ is small, which is consistent with our philosophy that the size of $\text{FR}(f)$ indicates the degree of randomness of the signal. To this point, the indicator of ${\mathbb Z}_p$ inside ${\mathbb Z}_{pq}$ has low complexity in any reasonable sense. However, if $p$ is much smaller than $q$, then $\text{FR}(f)$ is rather large, and the paradigm appears to break down.

The key observation here is that while $\text{FR}(f)=\sqrt{\frac{N}{p}}$, we have $\text{FR}(\widehat{f})=\sqrt{\frac{N}{q}}$. Why is $\text{FR}(\widehat{f})$ relevant? Simply because we can go from $f$ to $\widehat{f}$ by applying the Fourier matrix, and go back by applying the inverse Fourier matrix. Since the ${\mathbb Z}_N$ Fourier transform (or inverse Fourier transform) runtime is $\approx N \log N$ (see e.g., \cite{Bluestein1970, CooleyTukey1965, Good1958, Rader1968}), passing between $f$ and $\widehat{f}$ has relatively low computational cost, so low complexity of $f$ and low complexity of $\widehat{f}$ should be regarded as essentially equivalent from an algorithmic point of view.

In this way, the quantity
\begin{equation} \label{eq:twosidedratio}
  \text{FR}_{\mathrm{bi}}(f)
  = \min \left\{ \text{FR}(f), \text{FR}(\widehat{f}) \right\}
\end{equation}
may be better at capturing the complexity of a given signal in view of the discussion above.

\vskip.125in

\subsubsection{Inverse Fourier transform of a random set}
Let $S$ be a (uniformly) randomly chosen subset of $\mathbb{Z}_N$ of a given size, and let $f$ denote the inverse Fourier transform of $1_S$. Then
\[
  \text{FR}(f)=\frac{|S|}{\sqrt{|S|}}=\sqrt{|S|}.
\]

If $|S|$ is small, $\text{FR}(f)$ is small. This is completely consistent with Theorem~\ref{thm:L2polynomialapprox} and the other results in Subsection~\ref{subsection:main}, which say that a small Fourier ratio implies that the signal can be well-approximated by a trigonometric polynomial of small degree. It is also consistent with Theorem~\ref{theorem:alg}, which shows that we can effectively impute the values of such a signal from a small number of samples.

At first sight, this may seem counterintuitive, since for a typical random choice of $S$ the resulting time-domain signal $f$ has no obvious structure. However, from the point of view of compressed sensing, this is one of the simplest situations covered by Theorem~\ref{theorem:alg}. Indeed, the Fourier transform of $f$ is exactly sparse with random support, which is a very favorable configuration for recovery from few samples. Forecasting such a signal into the future is a completely different matter, and requires additional structural assumptions beyond small Fourier ratio. We shall return to this issue in the sequel.

\vskip.125in

\subsection{Approximation by low degree polynomials}
\label{subsection:main}

Our first main result amplifies the idea described in Subsection \ref{subsection:firstappearance}, which indicates that in the presence of randomness, the Fourier Ratio is large. Indeed, we will see that if the signal is even concentrated in a random set, then the ratio $\text{FR}(f)$ is suitably large. 

\begin{theorem} \label{thm:concentration} Let $f: {\mathbb Z}_N \to {\mathbb C}$. Suppose that there exists a generic set $M$ such that
$$ {\|f\|}_{L^2(M^c)} \leq r {\|f\|}_2$$ for some $r \in (0,1)$, with $|M| \leq \gamma_0 \frac{N}{\log(N)}$, where $\gamma_0$ is as in Theorem \ref{thm: combo}. Suppose that 
\begin{equation} \label{eq:concentrationrestriction} r< \frac{1-r}{C_T \sqrt{\log(N) \log \log(N)}}. \end{equation}

Then 
\begin{equation} \label{eq:largeconcentrationratio} \text{FR}(f) \ge \sqrt{N} \cdot \frac{1-r\frac{C_T \sqrt{\log(N) \log \log(N)}}{1-r}}{\frac{C_T \sqrt{\log(N) \log \log(N)}}{1-r}}\end{equation} with probability $1-o_N(1)$. 
\end{theorem} 

\vskip.125in 

\begin{remark} Using the observation in Remark \ref{rmk:nolog}, if $|M| \leq N^{1-\varepsilon}$ for some $\varepsilon>0$, we can replace (\ref{eq:concentrationrestriction}) with 
$$ r<\frac{1-r}{C_T},$$ and we can replace (\ref{eq:largeconcentrationratio}) with 
$$ \text{FR}(f) \ge \sqrt{N} \cdot \frac{\left(1-r\frac{C_T }{1-r} \right)}{\frac{C_T }{1-r}}.$$

\end{remark} 

\vskip.125in 

Our next result concerns $L^\infty$ approximations and it shows that if the corresponding ratio $\frac{\|\widehat f\|_{L^1(\mu)}}{\|f\|_\infty}$ is small, then we can approximate $f$ in $L^\infty$ norm by a trigonometric polynomial with degree logarithmic in $N$. 

\begin{theorem}\label{thm:Linftypolynomialapprox}
    Let $f:\mathbb Z_N\to\mathbb C$ and let $\eta>0$. Then for any integer $k$ such that
    $$k > 8\left(\frac{\|\widehat f\|_{L^1(\mu)}}{\|f\|_\infty}\right)^2 \frac{N\log (4N)}{\eta^2},$$
    there exists a trigonometric polynomial
    $$P(x) = \sum_{i=1}^k c_i \chi(m_ix)$$
    such that
    $$\|f - P\|_\infty < \eta \|f\|_\infty.$$
\end{theorem}

\begin{remark}
    Note that the triangle inequality shows $\frac{\|\widehat f\|_{L^1(\mu)}}{\|f\|_\infty} \geq N^{-\frac12}$ (with the equality for a dirac Delta $f=1_{\{x_0\}}$), and so in the best case, Theorem \ref{thm:Linftypolynomialapprox} indeed gives a polynomial of degree $O(\log(N))$.
\end{remark}

\vskip.125in 

\begin{remark} It is interesting to note that a slight variant of the quantity arising above, namely $\mu(f)=\frac{N\|f\|_\infty^2}{\|f\|_2^2}\in[1,N]$, is known as coherence, and it was introduced in a related context by Candes and Recht (\cite{CandesRecht2009}). This quantity is going to come up several times in this paper. 
\end{remark} 

\vskip.125in 

The proof of Theorem \ref{thm:Linftypolynomialapprox} is probabilistic, and similar techniques allow for good $L^2$ polynomial approximations in terms of the ratio $\text{FR}(f)$, as in the following theorem.

\begin{theorem}\label{thm:L2polynomialapprox}
    Let $f:\mathbb Z_N \to \mathbb C$, and let $\eta>0$. Then for any $k$ such that
    $$k > \frac{\text{FR}(f)^2 - 1}{\eta^2},$$
    there exists a trigonometric polynomial
    $$P(x) = \sum_{i=1}^k c_i \chi(m_ix)$$
    such that
    $$\|f - P\|_2 < \eta \|f\|_2.$$
\end{theorem}

\vskip.125in 

\begin{remark} Theorem \ref{thm:L2polynomialapprox} can be interpreted as follows. Define 
\begin{equation} \label{eq:fourierrationclass} \mathcal{C}(r) := \left \{f: \text{FR}(f) \leq r \right\}. \end{equation} 

Fix $\varepsilon > 0$ and suppose $f \in \mathcal{C}(r)$.  Then there exists a degree $\frac{r^2}{\varepsilon^2}$ trigonometric polynomial $P$ such that
$$ \|f-P\|_2 < \varepsilon \cdot \|f\|_2. $$

In other words, the class of signals with a small Fourier ratio is a subset of a set of signals of low complexity, as measured by the degree of the approximating polynomial. We shall build on this theme in the next section, where we use the notion of Kolmogorov complexity.           
\end{remark} 

\vskip.125in 

We can continue to use probabilistic techniques to get good $L^1$ approximations, now in terms of the corresponding ratio $\frac{\|\widehat f\|_1}{\|f\|_1}$.

\begin{theorem} \label{thm:L1polynomialapprox}
    Let $f:\mathbb Z_N \to \mathbb C$ and let $\eta>0$. If
    $$ k > 32\pi {\left( \frac{{\|\widehat{f}\|}_1}{{\|f\|}_1} \right)}^2 \cdot \frac{N}{\eta^2}, $$ then there exists a trigonometric polynomial $P$ of degree $k$ 
    such that
    $$ \|f-P\|_1 < \eta\|f\|_1.$$
\end{theorem} 

\begin{remark}
    We can again use the triangle inequality to see that $\frac{\|\widehat f\|_1}{\|f\|_1} \geq \frac{1}{\sqrt N}$, and consequently, in the best case, Theorem \ref{thm:L1polynomialapprox} gives a polynomial of constant degree. Notice, in the probability-normalized notation, the same argument yields, i.e., $\|\widehat f\|_{L^1(\mu)} \geq \frac{1}{\sqrt N} \|f\|_{L^1(\mu)}$. 
\end{remark}

\vskip.125in 

\subsection{Connections with the theory of Kolmogorov complexity} \label{subsection:CKScomplexity} We shall need the following definition. See e.g. \cite{LV19}, Chapter 2 and Chapter 7. 

\begin{definition} \label{def:kolmogorovcomplexity} Let \(U\) be a fixed universal Turing machine.  
For a function \( f : \mathbb{Z}_N \to \mathbb{C} \), its Kolmogorov complexity \(K(f)\) is defined as
\[
K(f) \;=\; \min \{ \, |p| \;:\; U(p, x) = f(x) \ \text{for all } x \in \mathbb{Z}_N \,\},
\]
where \(|p|\) denotes the length of the binary program \(p\).
\end{definition} 

Since we aim to understand to what extent the Fourier ratio allows us to determine how well we can expect to learn a given time series, the Kolmogorov complexity is too restrictive because it is easily led astray by small perturbations. To be precise, let $p,q$ denote two odd primes, and let $E=q{\mathbb Z}_p$, embedded in ${\mathbb Z}_{pq}$ in the natural way. As we saw in (\ref{eq:frforzp}), $\text{FR}(1_E)=\frac{\sqrt{N}}{\sqrt{|E|}}=\frac{\sqrt{N}}{\sqrt{p}} =\sqrt{q}.$ In the case when  $p$ is much larger than $q$, this ratio is rather small, and correctly suggests that $1_E$ is easy to learn. However, suppose that we modify $1_E$ by changing one of its values from $1$ to $1+\delta$, where $\delta$ is a very small positive number with $k$ digits in its base $2$ expansion. It follows that 
$$ K(f)=K(1_E)+O(k)$$ since it is not difficult to see that it takes $O(k)$ lines of a fixed universal Turing code to describe $\delta$. It is not difficult to see that up to a small error, $K(1_E)=\log(pq)$, so if $k$ is much larger than $\log(pq)$, the Kolmogorov complexity becomes arbitrarily large. And yet, $f$ is very well-approximated by a function with low Kolmogorov complexity. To deal with this issue, we introduce a refined version of Kolmogorov complexity (see e.g \cite{VV10}).  

\newcommand{\K}{\mathrm{K}}

Given $f : \mathbb{Z}_N \to \mathbb{C}$, the algorithmic rate--distortion function of $f$ at distortion level $\varepsilon > 0$ is defined by
\begin{equation} \label{eq:algorithmicratedistortion}
r_f(\varepsilon) \;=\; \min\{\;\K(g) : g : \mathbb{Z}_N \to \mathbb{C},\;
\|f-g\|_2 \leq \varepsilon\|f\|_2 \;\},
\end{equation} where $\K(g)$ is the Kolmogorov complexity of $g$ with respect to a fixed universal Turing machine.  Thus $r_f(\varepsilon)$ quantifies the minimum description length of an approximation to $f$ within distortion $\varepsilon$.

We can use Theorem \ref{thm:L2polynomialapprox} to show that functions with small $\text{FR}(f)$ must have small algorithmic rate-distortion, by showing that low-degree trigonometric polynomials have small Kolmogorov complexity. Our next result is the following. 

\begin{theorem}\label{thm:FourierRatio and AlgorithmicRateDistortion}
    Suppose $f:\mathbb{Z}_N\to\mathbb{C}$, and let $\varepsilon>0$. If $k = \frac{\text{FR}(f)^2}{\varepsilon^2}$, then we have that
    $$r_f(2\varepsilon) \leq C_Uk \log\left( \frac{(1+\varepsilon)N\sqrt k}{\varepsilon} \right) + C_U',$$
    where $C_U$ and $C_U'$ are constants depending only on the fixed universal Turing machine $U$.
\end{theorem}

We note that for $f:\mathbb{Z}_N\to\mathbb{C}$, $\K(f) \geq C\log N$ for some constant $C$ depending on the Turing machine $U$. Thus, we see that if $\text{FR}(f)$ is close to $1$, the algorithmic rate-distortion of $f$ is close to $\log N$, which is best possible.

\vskip.125in 

\subsection{Applications to the recovery of missing values in time series} \label{susbsection:imputation} We are now going to further exploit Theorem \ref{thm:L2polynomialapprox}. Suppose that $f: {\mathbb Z}_N \to {\mathbb C}$ is a signal, and that only the values $\{f(x): x \in X\}$ are known for some subset $X\subset \mathbb Z_N$.  The following result shows that if $f \in {\mathcal C}(r)$, the missing values can be recovered efficiently. 

\begin{theorem} \label{theorem:alg} Sample $q$ indices of $[N]$ uniformly and independently where 
$$q = Cr^2/\varepsilon^2 \log(r/ \varepsilon)^2 \log(N).$$  We name the sample set $X$.  If $C$ is a sufficiently large universal constant, then with high probability, for all $f \in \mathcal{C}(r)$, $x^{*}$, the solution to the linear program 
$$ \min_x \|\widehat{x}\|_1 \text{ subject to } \|f-x\|_{L^2(X)} \leq \|f\|_2 \cdot \varepsilon 
$$
satisfies
$$ 
\|x^*-f \|_2 \leq 11.47 \|f \|_2 \cdot \varepsilon.
$$
\end{theorem}

 Here,  the empirical $L^2(X)$-norm is defined by 
$ \|g\|_{L^{2}(X)}:=\big(\sum_{x\in X}|g(x)|^{2}\big)^{1/2} = \|g \cdot 1_{X}\|_2.$  

\vskip.125in 

\begin{remark} It should be noted that the optimization problem of Theorem \ref{theorem:alg} is a convex program, so it is algorithmically feasible. 
\end{remark}

A slight deficiency of Theorem \ref{theorem:alg} is the a priori assumption that $f \in {\mathcal C}(r)$. Since we are only given the values of $f$ restricted to $X$, it would be useful for us to have a result that says that the Fourier Ratio is not likely to be substantially different if it is computed over $X$ rather than over the whole ${\mathbb Z}_N$. This is what our next result is about. 

\begin{theorem}[Random restriction preserves $\text{FR}(f)$] \label{theorem:randomfourierration} Let $X\subset \mathbb Z_N$ be created by keeping each $x$ independently with probability $p\in(0,1]$, and define the restriction
$$
f_X(x)=\begin{cases}f(x), & x\in X,\\ 0,& x\notin X.\end{cases}
$$
Consider 
$$
\mu(f)=\frac{N\|f\|_\infty^2}{\|f\|_2^2}\in[1,N],
$$
and fix parameters $\varepsilon\in(0,1/2)$ and $u\ge 1$. There is a universal constant $C>0$ such that if
$$
p \ge C\,\frac{\mu(f)}{\varepsilon^2}\,\frac{\log N + u}{N},
$$
then with probability at least $1-2e^{-u}$ the following estimates hold: 
\begin{align}
\big|\|\widehat{f_X}\|_2-\sqrt{p}\,\|\widehat{f}\|_2\big| \le \varepsilon \sqrt{p}\,\|\widehat{f}\|_2,
\\
\big|\|\widehat{f_X}\|_1-p\,\|\widehat{f}\|_1\big| \le \varepsilon p\,\|\widehat{f}\|_1.
\end{align}

It follows that 
$$ \frac{\text{FR}(f_X)}{\text{FR}(f)} \in \left[\frac{1-\varepsilon}{\sqrt{1+\varepsilon}},\ \frac{1+\varepsilon}{\sqrt{1-\varepsilon}}\right]
\subseteq [1-3\varepsilon,\,1+3\varepsilon]. $$

In particular, 
$$ |\text{FR}(f_X)-\text{FR}(f)| \le 3\varepsilon\,\text{FR}(f). $$
\end{theorem}

\vskip.125in 

\begin{remark} Observe that $\mu(f) \leq \text{FR}(f)^2$, so Theorem \ref{theorem:randomfourierration} allows for restriction to random sets of size $\gtrsim \frac{r^2\log(N)}{\varepsilon^2}$. \end{remark} 

\vskip.125in 

We now address the fact that we only know ${||f_X||}_2$, not ${||f||}_2$. 

\begin{theorem} \label{theorem:sampledL2}
Let $N\ge 2$ and let $f:\mathbb Z_N\to\mathbb C$ be arbitrary. For $p\in(0,1]$, let $X\subset\mathbb Z_N$ be formed by keeping each $x\in\mathbb Z_N$ independently with probability $p$, and define $f_X(x)=f(x)$ if $x\in X$ and $f_X(x)=0$ otherwise. Let $\mu(f)$ be as above. 

Fix $\varepsilon\in(0,1)$ and $u\ge 1$. There is a universal constant $C>0$ such that if
$$
p \ge C\frac{\mu(f)}{\varepsilon^2}\frac{u}{N},
$$
then with probability at least $1-2e^{-u}$ we have
$$
|\|f_X\|_2-\sqrt p\,\|f\|_2|\le \varepsilon\sqrt p\,\|f\|_2.
$$
Equivalently,
$$
(1-\varepsilon)\sqrt p\,\|f\|_2 \le \|f_X\|_2 \le (1+\varepsilon)\sqrt p\,\|f\|_2.
$$
\end{theorem}

\vskip.125in 

\begin{remark} In view of Theorem \ref{theorem:sampledL2} and Theorem \ref{theorem:randomfourierration}, we can replaced $\text{FR}(f)$ and ${\|f\|}_2$ in Theorem \ref{theorem:alg} with suitably scaled $\text{FR}(f_X)$ and ${\|f_X\|}_2$, which resolves the data leakage problem in the sense that we do not know $f$ away from $X$. 
\end{remark} 

\vskip.125in

\subsection{VC Dimension and Statistical Query Dimension} \label{subsection:dimensions}

We are now going to conceptualize these results further by bringing in the concepts of the Vapnik-Chervonenkis and statistical query dimension. 

\begin{definition}
For a concept class, $\mathcal{C} \subset \{-1,1\}^X$, and probability distribution, $\mathcal{D}$, on $X$, the statistical query
dimension of $\mathcal{C}$ with respect to $\mathcal{D}$ is the largest number $d$ such that $\mathcal{C}$ contains $d$
functions $f_1, f_2, . . . , f_d$ such that for all $i \not = j$,
$$
\left |\mathbb{E}_{x \sim \mathcal{D}}
\left
[f_i(x) \cdot f_j(x)
\right
] \right| \leq \frac{1}{d}.
$$
The Statistical Query Dimension of $\mathcal{C}$ is the maximum of the statistical query
dimension of $\mathcal{C}$ with respect to $\mathcal{D}$ over all $\mathcal{D}$.
\end{definition}

\begin{definition} Given a concept class $\mathcal{C} \subset \{-1,1\}^X$, we say that ${\mathcal C}$ shatters the points $c_1, c_2, \dots, c_n$ from $X$ if the restriction of ${\mathcal C}$ to $\{c_1, c_2, \dots, c_n\}$ is the set of all functions from $\{c_1, c_2, \dots, c_n\}$ to $\{-1,1\}$. We say that the VC dimension of ${\mathcal C}$ is equal to $d$ if ${\mathcal C}$ shatters some set of $d$ points, but it does not shatter any set with $d+1$ points. 
\end{definition} 

\vskip .125in

The following result from \cite{Reyzin20} says that the statistical query dimension bounds the VC dimension.
\begin{theorem}
\label{thm:VCdimtoSQdim}
Let $\mathcal{C} \subset \{-1,1\}^X$.  If the VC-dimension of $\mathcal{C}$ is $d$, then the Statistical Query dimension of $\mathcal{C}$ is $\Omega(d)$, where here, and throughout, $A=\Omega(B)$ if there exists a finite $c>0$ such that $A \ge cB$. \end{theorem}

\vskip .125in

Next, we consider the statistical query dimension and VC-dimension of the class of functions with bounded Fourier Ratio. In order to match the above setting, define
$$\mathcal B(r) = \{f:\mathbb Z_N \to \{-1,1\} \,:\, \text{FR}(f) \leq r\} \subset \{-1,1\}^{\mathbb Z_N}.$$
The next result bounds the statistical query dimension and VC-dimension of $\mathcal{B}(r)$.
\begin{theorem}
\label{thm:VC_Result_New}
The statistical query dimension of $\mathcal{B}(r)$ is 
\begin{equation} \label{eq:dimensionprecise}  \leq \frac{1}{\varepsilon^{\frac{r^2}{\varepsilon^2}}} \cdot \left({\frac{e}{\frac{r^2}{\varepsilon^2}}}\right)^{\frac{r^2}{\varepsilon^2}} N^{\frac{r^2}{\varepsilon^2}},\end{equation} where $r=\text{FR}(f)$ and $\varepsilon=\frac{1}{4(1+2r)}$. In particular,  the dimension is  
 \[
\leq \left(\frac{eN}{4 r^{2} (1+2r)}\right)^{\frac{r^{2}}{\varepsilon^{2}}}
=
\left(\frac{eN}{4 r^{2} (1+2r)}\right)^{16 r^{2} (1+2r)^{2}},
\qquad
\text{since } \frac{r^{2}}{\varepsilon^{2}} = r^{2}\cdot 16(1+2r)^{2}.
\]
\end{theorem}

\vskip.125in 

\begin{remark} The VC-dimension of $\mathcal{B}(r)$ can be estimated from (\ref{eq:dimensionprecise}) using Theorem \ref{thm:VCdimtoSQdim}. \end{remark} 

\vskip.125in 

\vskip .25in 

\subsection{Stability of the Fourier ratio} \label{subsection:stability} We are now going to investigate how the Fourier Ratio $\text{FR}(f)$ behaves under small perturbations. 

    \begin{theorem} \label{theorem:perturbation} Let $f:\mathbb Z_N \to \mathbb C$. For any perturbation $n: {\mathbb Z}_N \to {\mathbb C}$, set $A=\|\widehat f\|_1$, $B=\|\widehat n\|_1$, $s=\|\widehat f\|_2$, $t=\|\widehat n\|_2$. If $t<s$, then
\begin{equation} \label{eq:perturbation}
\big|\text{FR}(f+n)-\text{FR}(f)\big|
\le
\frac{\|\widehat n\|_1 + \text{FR}(f)\|\widehat n\|_2}{\|\widehat f\|_2 - \|\widehat n\|_2}.
\end{equation} 
Moreover, using $\|\widehat n\|_1 \le \sqrt{N}\|\widehat n\|_2$ and $\text{FR}(f) \le \sqrt{N}$, we have
\[
\big|\text{FR}(f+n)-\text{FR}(f)\big|
\le \frac{\left(\sqrt{N}+\text{FR}(f)\right)\|\widehat n\|_2}{\|\widehat f\|_2-\|\widehat n\|_2}
\le \frac{2\sqrt{N}\|\widehat n\|_2}{\|\widehat f\|_2-\|\widehat n\|_2}.
\]
In particular, if $\|\widehat n\|_2 \le \frac12 \|\widehat f\|_2$, then
\[
\big|\text{FR}(f+n)-\text{FR}(f)\big|
\le 2\left(\frac{\|\widehat n\|_2}{\|\widehat f\|_2}\right)\left(\sqrt{N}+\text{FR}(f)\right).
\]
\end{theorem}

\vskip.125in 

One way of interpreting Theorem \ref{theorem:perturbation} is the following. Let $\delta=\frac{{||\widehat{n}||}_2}{{||\widehat{f}||}_2}$. Then (\ref{eq:perturbation}) tells us that 
$$ |\text{FR}(f+n)-\text{FR}(f)| \leq \frac{\delta}{1-\delta} \cdot \left(\text{FR}(f)+\text{FR}(n) \right). $$

If $\frac{\delta}{1-\delta} \leq \frac{1}{\sqrt{N}}$, we obtain 
$$ |\text{FR}(f+n)-\text{FR}(f)| \leq \frac{1}{\sqrt{N}} \text{FR}(f)+1.$$

\vskip.125in 

Our next result tells us what happens if $n$ is the Gaussian noise with a given variance. 

\begin{theorem} \label{theorem:gaussperturbation}
Let $f,n:\mathbb{Z}_N \to \mathbb{C}$. Assume $n(x)$ are independent circular complex Gaussian variables with variance $\sigma^2$, and fix $\gamma \in(0,1)$. With probability at least $1-\gamma$,
\[
\left|\text{FR}(f+n)-\text{FR}(f)\right|
\le
\frac{\left(\text{FR}(n)+\text{FR}(f)\right)\,t_\gamma}{\|\widehat f\|_2 - t_\gamma}, \] where
\qquad
\[ t_\gamma=\sigma\left(\sqrt{N}+\sqrt{\log(1/\gamma)}\right).\]
\end{theorem}

\vskip.125in 

We will now examine the situation where $f_i: {\mathbb Z}_N \to {\mathbb C}$, $1 \leq i \leq n$ is a collection of time series such that $f_i(x)=s(x)+n_i(x)$, where $s(x)$ is the "true" (unknown) signal, and $n_i$ is Gaussian noise. The following well-known result shows that the average of $f_i$ is, with high probability, much closer to the true signal $s(x)$ than any of the individual $f_i$s. 

\begin{theorem}[Smoothing by averaging independent Gaussian noise] \label{theorem:smoothingeffect}
Let $f_1,\dots,f_n:\mathbb{Z}_N\to\mathbb{C}$ be given by $f_i(x)=s(x)+n_i(x)$, where $s:\mathbb{Z}_N\to\mathbb{C}$ is an unknown signal and, for each $i$ and $x$, the noise $n_i(x)$ is circular complex Gaussian with mean $0$ and variance $\sigma^2$, independent across $i$ and $x$. Define the average
\[
f(x)=\frac{1}{n}\sum_{i=1}^n f_i(x)\qquad(x\in\mathbb{Z}_N).
\]
Then:

(a) Unbiasedness and pointwise variance reduction:
\[
\mathbb{E}\,f(x)=s(x),\qquad
f(x)-s(x)\sim \mathcal{CN}\!\left(0,\frac{\sigma^2}{n}\right)
\text{ for each }x.
\]

(b) Mean squared error in $\ell^2$:
\[
\mathbb{E}\,\|f-s\|_2^2=\frac{N\sigma^2}{n},
\quad\text{where }\|h\|_2^2=\sum_{x\in\mathbb{Z}_N}|h(x)|^2.
\]

(c) High-probability bound (norm concentration): for any $\gamma\in(0,1)$, with probability at least $1-\gamma$,
\[
\|f-s\|_2\le \frac{\sigma}{\sqrt{n}}\Big(\sqrt{N}+\sqrt{\log(1/\gamma)}\Big).
\]
A slightly sharper but longer bound is
\[
\|f-s\|_2\le \frac{\sigma}{\sqrt{n}}\Big(\sqrt{N}+\sqrt{\log(1/\gamma)}+\frac{\log(1/\gamma)}{2\sqrt{N}}\Big)
\quad\text{with probability at least }1-\gamma.
\]

\end{theorem}

\vskip.125in 

\begin{remark} Parts (a)-(b) of Theorem \ref{theorem:smoothingeffect} are proved in \cite{KayEstimation,LehmannCasella,Wasserman} (and discussed in \cite{OppenheimSchafer,Hayes}); the concentration ingredient for part (c) is proved in \cite{Vershynin2018,BLM2013} (Lipschitz Gaussian concentration) or \cite{LaurentMassart2000,Wainwright2019} (chi–square tails).
\end{remark} 

\vskip.125in 

Our next result analyzes the relationship between the Fourier Ratio of the average $f$ and the true signal $s$. We are going to see that averaging brings the Fourier ratio toward the signal at an $\frac{1}{\sqrt{n}}$ rate. 

\begin{theorem} \label{theorem:FRofaverage}
Let $f_i:\mathbb{Z}_N\to\mathbb{C}$ be given by $f_i=s+n_i$ for $i=1,\dots,n$, where $s:\mathbb{Z}_N\to\mathbb{C}$ is fixed and the noises $n_i(x)$ are independent across $i$ and $x$, each circular complex Gaussian with mean $0$ and variance $\sigma^2$. Define the average
\[
f=\frac{1}{n}\sum_{i=1}^n f_i = s+\overline n,\qquad \overline n:=\frac{1}{n}\sum_{i=1}^n n_i.
\]

Fix $\gamma\in(0,1)$ and set $r_\gamma=\sqrt{N}+\sqrt{\log(1/\gamma)}$. Then the following hold.

\vskip.125in 

(i) With probability at least $1-\gamma$,
\[
\|\overline n\|_2 \le \frac{\sigma}{\sqrt{n}}r_\gamma,
\qquad
\|n_i\|_2 \le \sigma r_\gamma\quad\text{for each fixed }i.
\]
Consequently, if $\|s\|_2\ge 2\sigma r_\gamma$, then with probability at least $1-\gamma$,
\[
\big|\text{FR}(f)-\text{FR}(s)\big| \le \frac{2\sigma r_\gamma}{\sqrt{n}\,\|s\|_2}\big(\text{FR}(\overline n)+\text{FR}(s)\big),
\]
and for each fixed $i$,
\[
\big|\text{FR}(f_i)-\text{FR}(s)\big| \le \frac{2\sigma r_\gamma}{\|s\|_2}\big(\text{FR}(n_i)+\text{FR}(s)\big).
\]

\vskip.125in 

(ii) In particular, on the event in (i), the averaged estimate satisfies
\[
\big|\text{FR}(f)-\text{FR}(s)\big| \le \frac{1}{\sqrt{n}} C_\gamma(s,n_i,\overline n),
\]
where
\[
C_\gamma(s,n_i,\overline n)=\frac{2\sigma r_\gamma}{\|s\|_2}\big(\text{FR}(\overline n)+\text{FR}(s)\big),
\]
so the deviation of $\text{FR}(f)$ from $\text{FR}(s)$ is smaller by a factor $1/\sqrt{n}$ relative to the corresponding single-shot bound (up to the benign replacement of $\text{FR}(n_i)$ by $\text{FR}(\overline n)$, which has the same distribution because $\text{FR}(\alpha e)=\text{FR}(e)$ for all $\alpha>0$).

\end{theorem}

\vskip.25in 

\subsection{Chang's Lemma and additive structure} \label{subsection:chang} 

Next, we consider the extent to which the ratio $\text{FR}(f)$ indicates an additive structure of the function $f$. For indicator functions, we have the following result, which shows that large spectrum sets have some additive structure.

\begin{theorem}[Sparse-spectrum approximation from small $\text{FR}(f)$]\label{thm:sparse-spectrum approximation theorem}
    Let $f:\mathbb Z_N \to \mathbb C$, and fix $\eta>0$. Define the large spectrum
    $$\Gamma := \left\{ m \in \mathbb Z_N \,:\, |\widehat f(m)| \geq \eta \|f\|_{L^2(\mu)}\right\}.$$
    Then
    $$|\Gamma| \leq \frac{\text{FR}(f)}{\eta}\sqrt N,$$ and moreover if
    $$P(x) := \frac{1}{N^\frac12} \sum_{m\in \Gamma} \widehat f(m) \chi(xm),$$ then
    $$\|f-P\|_2 \leq \eta \|f\|_2.$$
    
    In other words, $f$ can be approximated by a polynomial of degree 
    $$\frac{\text{FR}(f)}{\eta}\cdot \sqrt N$$ up to an error $\leq \eta$. 
\end{theorem}

\vskip.125in 

\begin{remark} Note that our bounds on $\text{FR}(f)$ show that in the best case, this approximating polynomial could be of degree $\frac{1}{\eta} \sqrt N$. While this is of worse degree than in Theorem \ref{thm:L2polynomialapprox}, in this case the approximating polynomial is deterministic. Moreover, we will see that we can ensure the Fourier support of the polynomial $P(x)$ has an additive structure by using a suitable generalization of a result due to Chang. \end{remark}

A result due to Chang (\cite{Chang02}) shows that large spectrum sets, as in Theorem \ref{thm:sparse-spectrum approximation theorem}, have additive structure in the case that $f$ is the indicator function of a set.

\begin{lemma}[Chang's lemma]{\label{lemma:changs}}    Let $A\subset \mathbb Z_N$ have density $\alpha = \frac{|A|}{N}$, and for $\eta>0$ define the large spectrum set
    $$\Gamma = \left\{ m\in\mathbb Z_N \,:\, |\widehat{1_A}(m)| \geq \eta \alpha N^\frac12 \right\}.$$
    Then there exists $\Lambda \subset \Gamma$ with
    $$|\Lambda| \leq C \eta^{-2} \log \left(\frac 1 \alpha \right)$$
    such that every $m\in \Gamma$ is a $\{-1,0,1\}$-linear combination of elements of $\Lambda$.
\end{lemma}

Lemma \ref{lemma:changs} was used by Chang originally to improve on bounds in Freiman's theorem (\cite{Chang02}), as well as by Green in \cite{Green02} to improve on a result due to Bourgain on arithmetic progressions in sumsets.

Chang's lemma suggests that large spectrum sets have some arithmetic structure, and by generalizing Chang's lemma from indicator functions to general functions $f:\mathbb Z_N\to \mathbb C$, we can show that a small $\text{FR}(f)$ implies additive structure in the large valued sets of a function. To this end, we have the following result.

\begin{theorem}[Generalized Chang's Lemma]{\label{thm: generalizedchang}}
    Let $f:\mathbb{Z}_N\to\mathbb{C}$, and for $\eta>0$ define the large spectrum set
    $$\Gamma=\left\{ m\in\mathbb Z_N \,:\, |\widehat{f}(m)| \geq \eta\|f\|_{L^2(\mu)}\right\}.$$
    Then there exists a constant $C$ and some $\Lambda\subset\Gamma$ such that every $m\in \Gamma$ is a $\{-1,0,1\}$-linear combination of elements of $\Lambda$, and moreover, \begin{align}\label{eq:generalizedchang lognorm}
        |\Lambda|&\leq C\eta^{-2}\left(\frac{\|f\|_\frac{\log N}{\log N-1}}{\|f\|_2}\right)^2\log N.
    \end{align}
    If, additionally, $\frac{\|f \|_1}{\|f\|_2} \leq \frac{1}{e}\sqrt N$, then we have
    \begin{align}\label{eq:generalizedchang l2l1}
        |\Lambda|&\leq C\eta^{-2}\frac{\|f\|_1^2}{\|f\|_2^2}\log\left(\left(\frac{\|f\|_2}{\|f\|_1}\right)^2N\right).
    \end{align}
\end{theorem}

\begin{remark}
    Applying Theorem \ref{thm: generalizedchang} to $\widehat{f}$, whenever $\text{FR}(f) \leq \frac{1}{e}\sqrt N$ we obtain
    \begin{align*}
    |\Lambda|\leq C\eta^{-2}\text{FR}(f)^2\log\left(\text{FR}(f)^{-2}N\right),
    \end{align*}from (\ref{eq:generalizedchang l2l1}), where $\Lambda$ is a set such that every\begin{align*}
        x\in\Gamma\coloneq\left\{x\in\mathbb{Z}_N:|f(x)|\geq\eta\|f\|_{L^2(\mu)}\right\}
    \end{align*} is a $\{-1,0,1\}$-linear combination of elements of $\Lambda$. Since the $\text{FR}(f)$ term outside the $\log$ dominates, this implies that a small $\text{FR}(f)$ means a small $\Lambda$, pointing towards additive structure in the large valued set of $f$. Alternatively, for large $N$, (\ref{eq:generalizedchang lognorm}) approaches\begin{align*}
        |\Lambda|\leq C'\eta^{-2}\text{FR}(f)^2\log N,
    \end{align*} 
    where $C'= Ce^{-2}$, giving the same result.
\end{remark}

\vskip.25in

\vskip.25in 

\section{Numerical experiments} \label{section:numerical}

\vskip.125in 

In this section, we describe numerical experiments to estimate the Fourier ratio for reasonable data sets and to obtain numerical evidence about Talagrand's constant. 

\subsection{Fourier ratio for real-life data sets}  

We will illustrate the utility of the ratio $\text{FR}(f)$ on real-life data sets. The plots of the data sets are in Figure \ref{fig:real_processes} below. The results are as follows: 
\begin{itemize}
    \item Peyton Manning Wikipedia Visits: \(\text{FR}(f) = 1.917\)
    \item Electric Production: \(\text{FR}(f) = 2.133\)
    \item Delhi Daily Climate: \(\text{FR}(f) = 2.715\)
    \item Australia Monthly Beer Production: \(\text{FR}(f) = 2.884\)
\end{itemize}

\begin{figure}[htbp]
\centering
\includegraphics[scale=0.5]{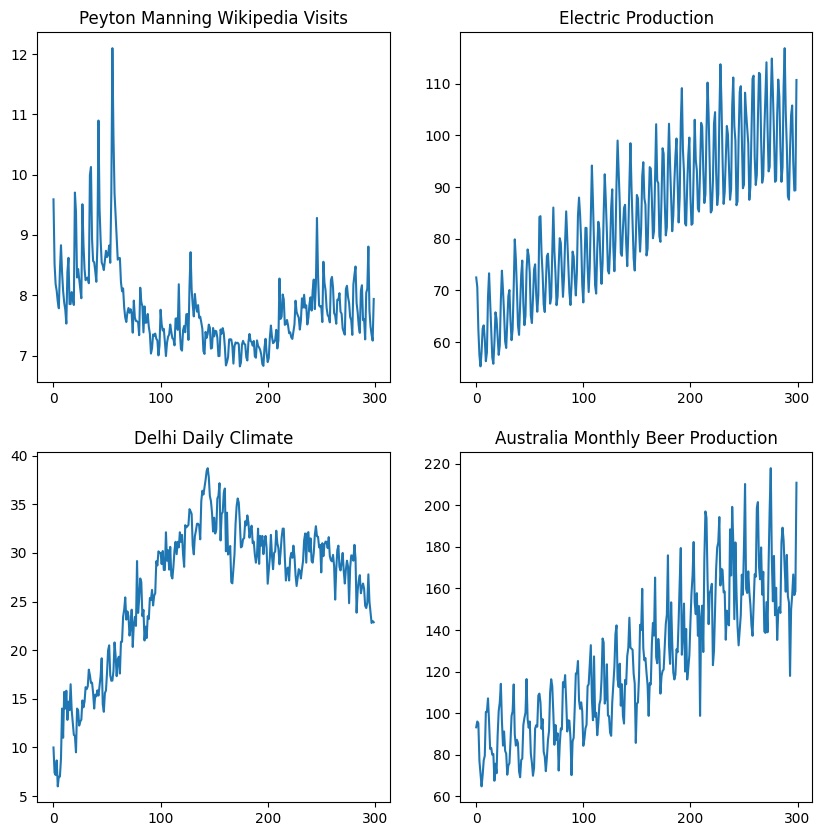}
\caption{Four real-world datasets: plots show raw values (not Fourier ratios), so the 
$y$-axis is the observed series value.}
\label{fig:real_processes}
\end{figure}

These data sets are unrelated to one another and are dominated by event-driven fluctuations, like football game times and global weather patterns. The fact that their Fourier ratios are all relatively close to the minimal value of $1$ gives heuristic credibility to the definition of randomness detection studied in this paper.  

\vskip.125in 

\subsection{Talagrand's Constant} 

In this section, we are exploring the constants $C_T$ (Talagrand's Constant) and $C(q)$ (Bourgain's Constant). In lieu of Theorem \ref{thm: combo}, there exists $\gamma_0 \in (0,1)$ such that if $h: {\mathbb Z}_N \to {\mathbb C}$ is supported in a generic set $M$ (in the sense of Definition \ref{def:generic}) of size $\gamma_0 \frac{N}{\log(N)}$, then with probability $1-o_N(1)$, 
\begin{equation} \label{eq: talagrand} 
{\left( \frac{1}{N}  \sum_{m \in {\mathbb Z}_N} {|\widehat{h}(m)|}^2 \right)}^{\frac{1}{2}} \leq C_T {(\log(N) \log \log(N))}^{\frac{1}{2}} \cdot \frac{1}{N} \sum_{m \in {\mathbb Z}_N} |\widehat{h}(m)|, 
\end{equation} 
where $C_T > 0$ is a constant that depends only on $\gamma_0$.  

In lieu of Remark \ref{rmk:nolog}, when $|M|=O(N^{1-\varepsilon})$ for some $\varepsilon>0,$ we may replace $C_T {(\log(N) \log \log(N))}^{\frac{1}{2}}$ with just $C'_T.$ So in this case, we will take Talagrand's constant to be $C'_T.$ We will write this as $C_T$ for the remainder of this section, as we will let $|M|=O(N^{\frac{2}{q}})$ so $\varepsilon=\frac{q-2}{q}$ for some $q>2.$ We will now write \eqref{eq: talagrand} as 
\begin{equation}\label{eq: talagrandprime}
    {\left( \frac{1}{N}  \sum_{m \in {\mathbb Z}_N} {|\widehat{h}(m)|}^2 \right)}^{\frac{1}{2}} \leq C_T \frac{1}{N} \sum_{m \in {\mathbb Z}_N} |\widehat{h}(m)|.
\end{equation}

Now we look at Theorem \ref{thm:bourgain}. Suppose that $M$ is generic, as above, $|M|= \lceil N^{\frac{2}{q}} \rceil$, $q>2$. Then for all $f: {\mathbb Z}_N \to {\mathbb C}$ supported in $M$, 
\begin{equation} \label{eq:bourgain} 
{\left( \frac{1}{N} \sum_{m \in {\mathbb Z}_N} {|\widehat{f}(m)|}^q \right)}^{\frac{1}{q}} \leq C(q) \cdot {\left( \frac{1}{N} \sum_{m \in {\mathbb Z}_N} {|\widehat{f}(m)|}^2 \right)}^{\frac{1}{2}}. 
\end{equation} 

Looking at Lemma \ref{lemma:vershynintrick}, we apply H\"older's inequality to \eqref{eq:bourgain} to see that 
\[
 {\left( \frac{1}{N} \sum_{x \in {\mathbb Z}_N} {|\widehat{f}(x)|}^2 \right)}^{\frac{1}{2}} \leq {(C(q))}^{\frac{q}{q-2}} \cdot \frac{1}{N} \sum_{x \in {\mathbb Z}_N} |\widehat{f}(x)|.
\]

Combining the above, we see that $C_T\leq C(q)^{\frac{q}{q-2}}.$ 
Recalling the notation
\begin{equation} \label{eq:probnorm}
    {\|g\|}_{L^p(\mu)}={\left( \frac{1}{N} \sum_{x \in {\mathbb Z}_N} {|g(x)|}^p \right)}^{\frac{1}{p}},
\end{equation}
we get that \eqref{eq: talagrandprime} becomes 
\begin{equation}\label{eq: talagrandprimenorm}
    {\|\widehat{h}\|}_{L^2(\mu)}\le C_T\|\widehat{h}\|_{L^1(\mu)}
\end{equation}
and \eqref{eq:bourgain} turns into 
\begin{equation}\label{eq: bourgainnorm}
     {\|\widehat{f}\|}_{L^q(\mu)}\le C(q)\|\widehat{f}\|_{L^2(\mu)}.
\end{equation}

Now, we want to perform numerical experiments to compute \(\frac{\|\widehat{f}\|_{L^q(\mu)}}{\|\widehat{f}\|_{L^2(\mu)}}\) and \(\frac{\|\widehat{f}\|_{L^2(\mu)}}{\|\widehat{f}\|_{L^1(\mu)}}\) so we may investigate whether $C_T\approx C(q)^{\frac{q}{q-2}}$ or whether $C_T$ is in fact much smaller. 

All of the following algorithms were implemented and tested in Python. We will also be computing the ratio of the norms using indicator functions on the random sets $M$. The values obtained are not for general functions $f.$ Instead, we will compute them when $f=\widecheck{1}_M.$ As these are random signals, there is expected to be some variance between the data for varying trials. In lieu of this, exact results will not be possible for evaluating these constants, but to counter this, we used large values of $N$ and one million trials. Although computationally expensive, this allowed us to obtain reasonably stable estimates. In all these algorithms, we computed the ratios with random sizes $M$. We took the $90$th percentile of all these trials as the estimate for the constants to avoid the far outliers. As for the Bourgain constant, the algorithm will compute the ratio $\frac{\|\widehat{f}\|_{L^q(\mu)}}{\|\widehat{f}\|_{L^2(\mu)}}$ and then raise this value to the $\frac{q}{q-2}$ power so we may compare $C_T$ and $C(q)^{\frac{q}{q-2}}.$ 

\subsection{Computing Talagrand/Bourgain Constants}

Here, we present 7 numerical experiments which pin down estimates for $C(q)$ and $C_T$.

In the first algorithm, we first sought to see the relationship between $C(q)^{\frac{q}{q-2}}$ and $q$. We fixed $N=100000$ and let $q\in[3,4].$ Letting $q\rightarrow 2$ would make the exponent blow up, so we avoided computing the constant close to $q=2.$

\begin{figure}[H]
    \centering
    \includegraphics[width=0.5\linewidth]{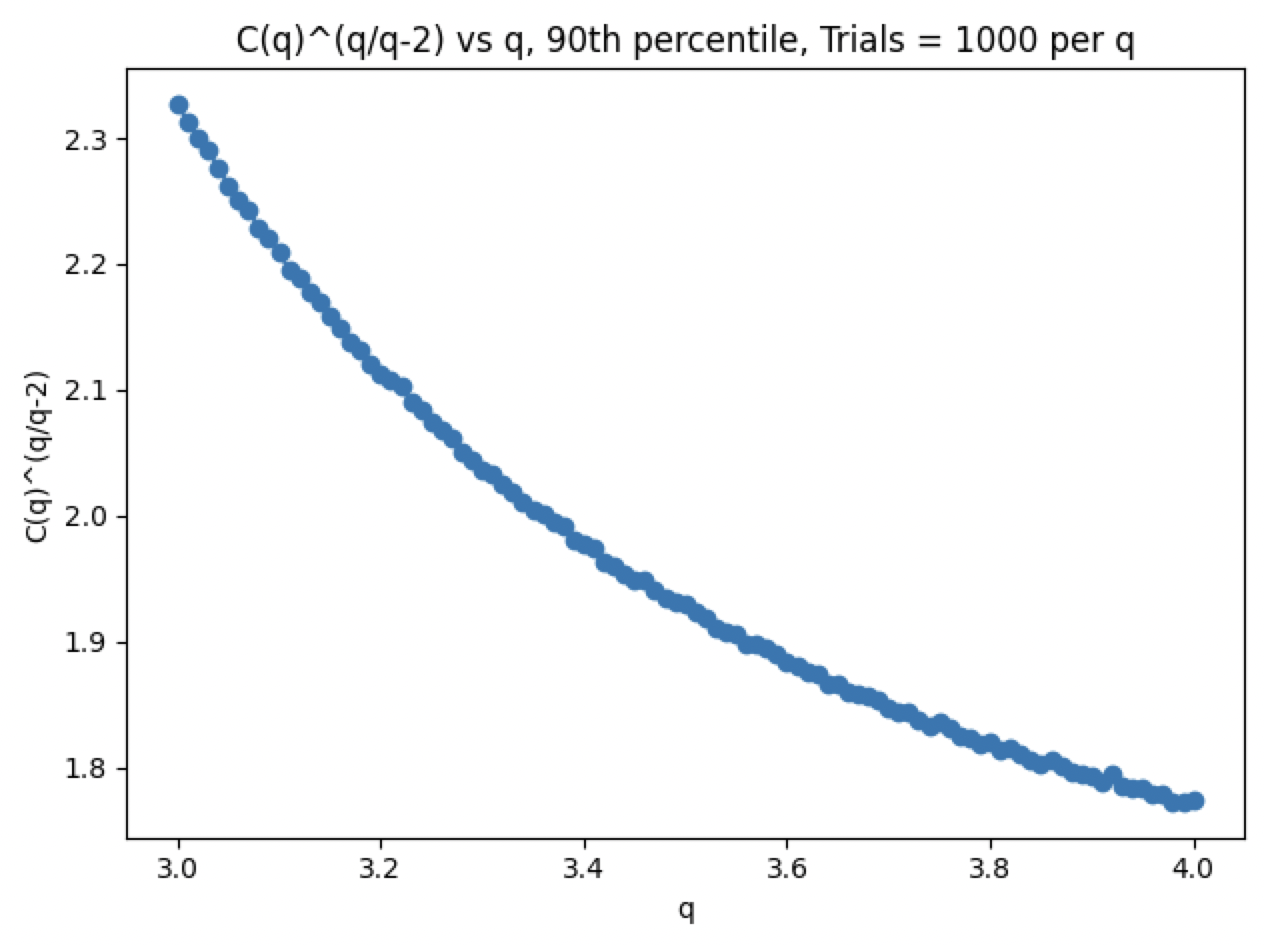}
    \caption{$C(q)^{\frac{q}{q-2}}$ vs $q$}
    \label{fig:c(q)expvsq}
\end{figure}

In this second algorithm, we sought to perform a sanity check to see that Bourgain's constant does not depend on $N.$ We fixed $q=4$ and let $N\in[10,100000].$

\begin{figure}[H]
    \centering
    \includegraphics[width=0.5\linewidth]{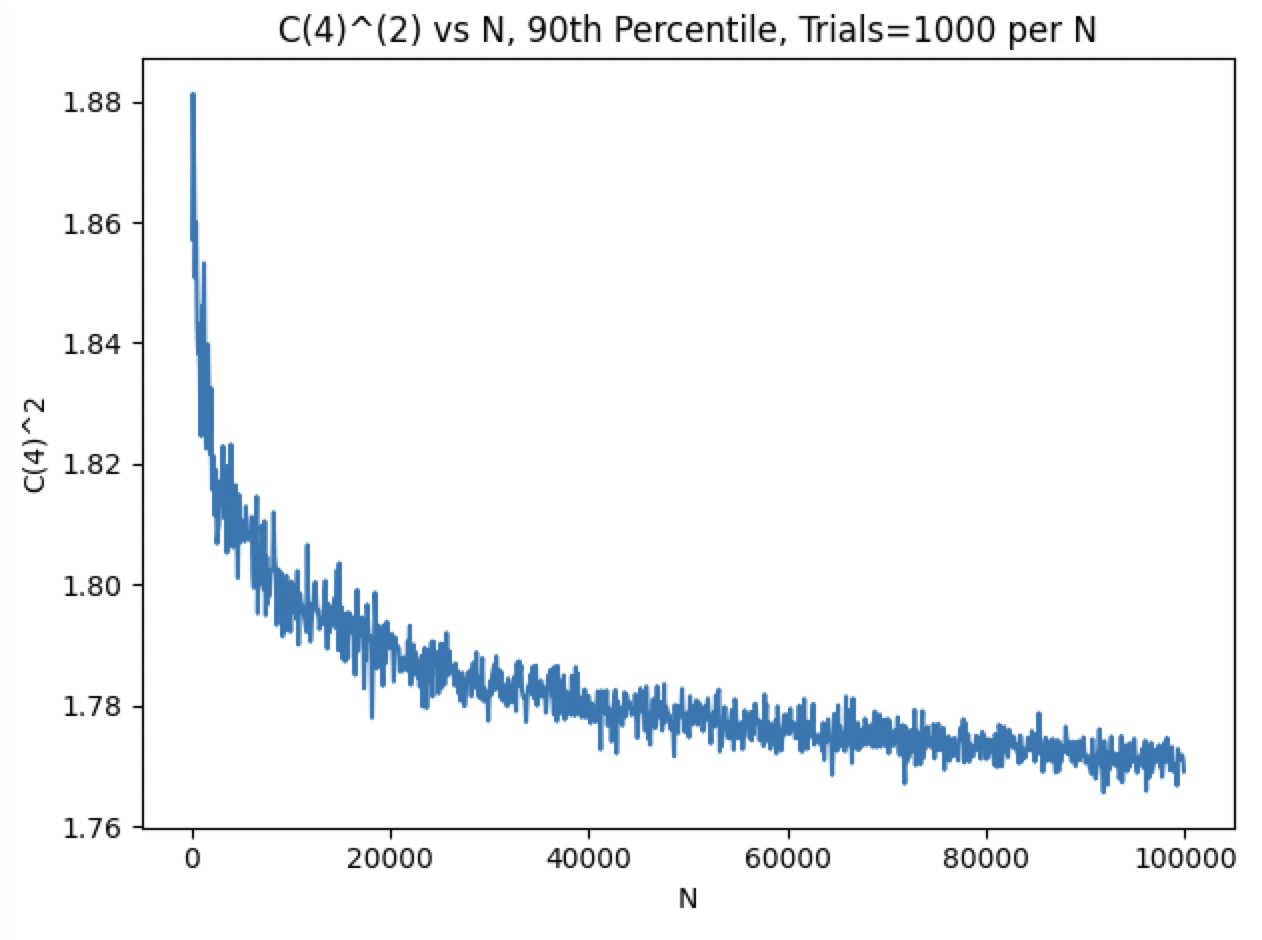}
    \caption{$C(q)^{\frac{q}{q-2}}$ vs $N$}
    \label{fig:C(q)expvsN}
\end{figure}

In this third algorithm, we created a heatmap for varying $q$ and $N$ to see how the two variables affect $C(q)^{\frac{q}{q-2}}.$

\begin{figure}[H]
    \centering
    \includegraphics[width=0.5\linewidth]{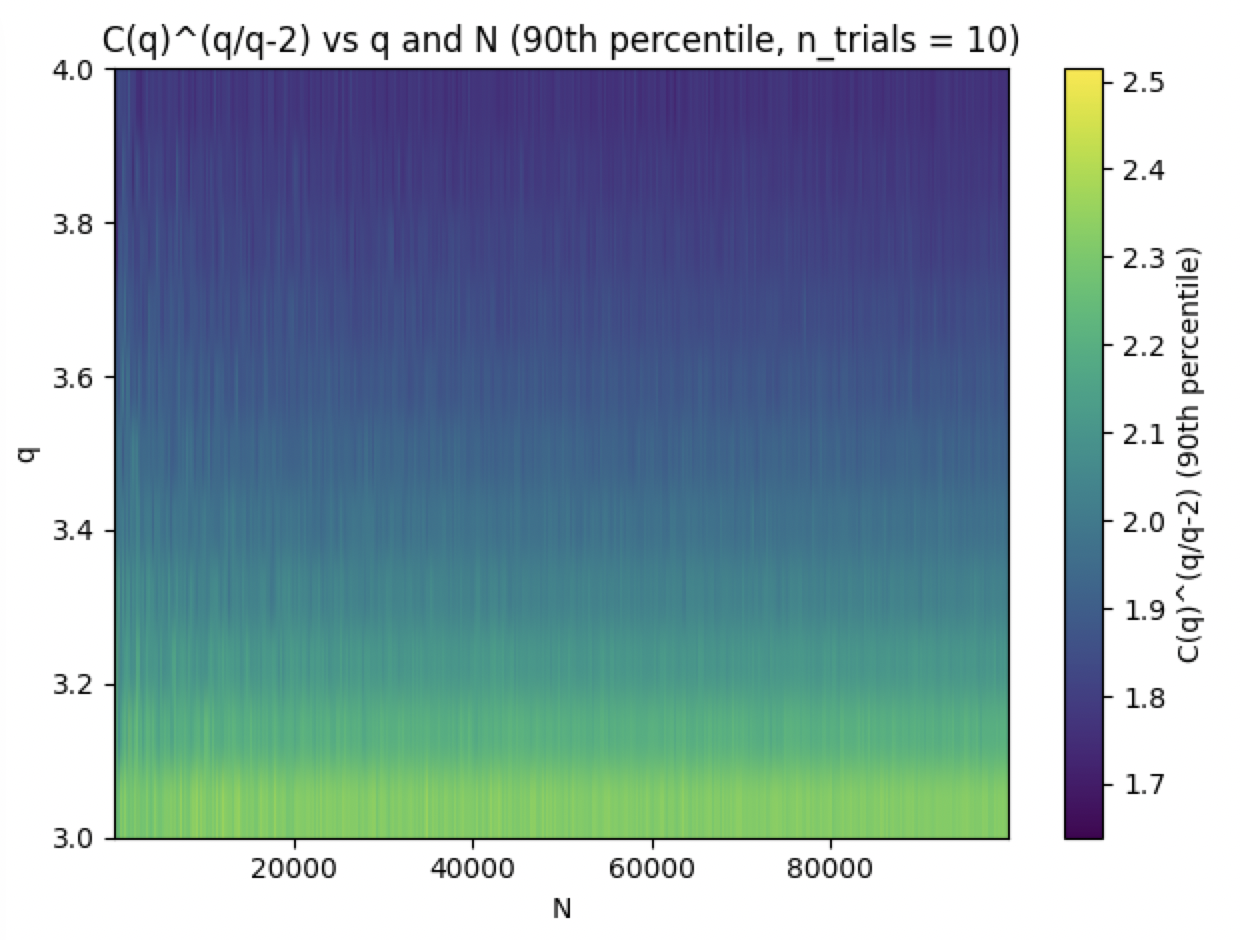}
    \caption{$C(q)^{\frac{q}{q-2}}$ vs $q$ and $N$}
    \label{fig:C(q)heat}
\end{figure}

In the fourth algorithm, we first sought to see a relationship between $C_T$ and $q$. We fixed $N=100000$ and let $q\in[3,4]$ to match up with Bourgain's constant.  

\begin{figure}[H]
    \centering
    \includegraphics[width=0.5\linewidth]{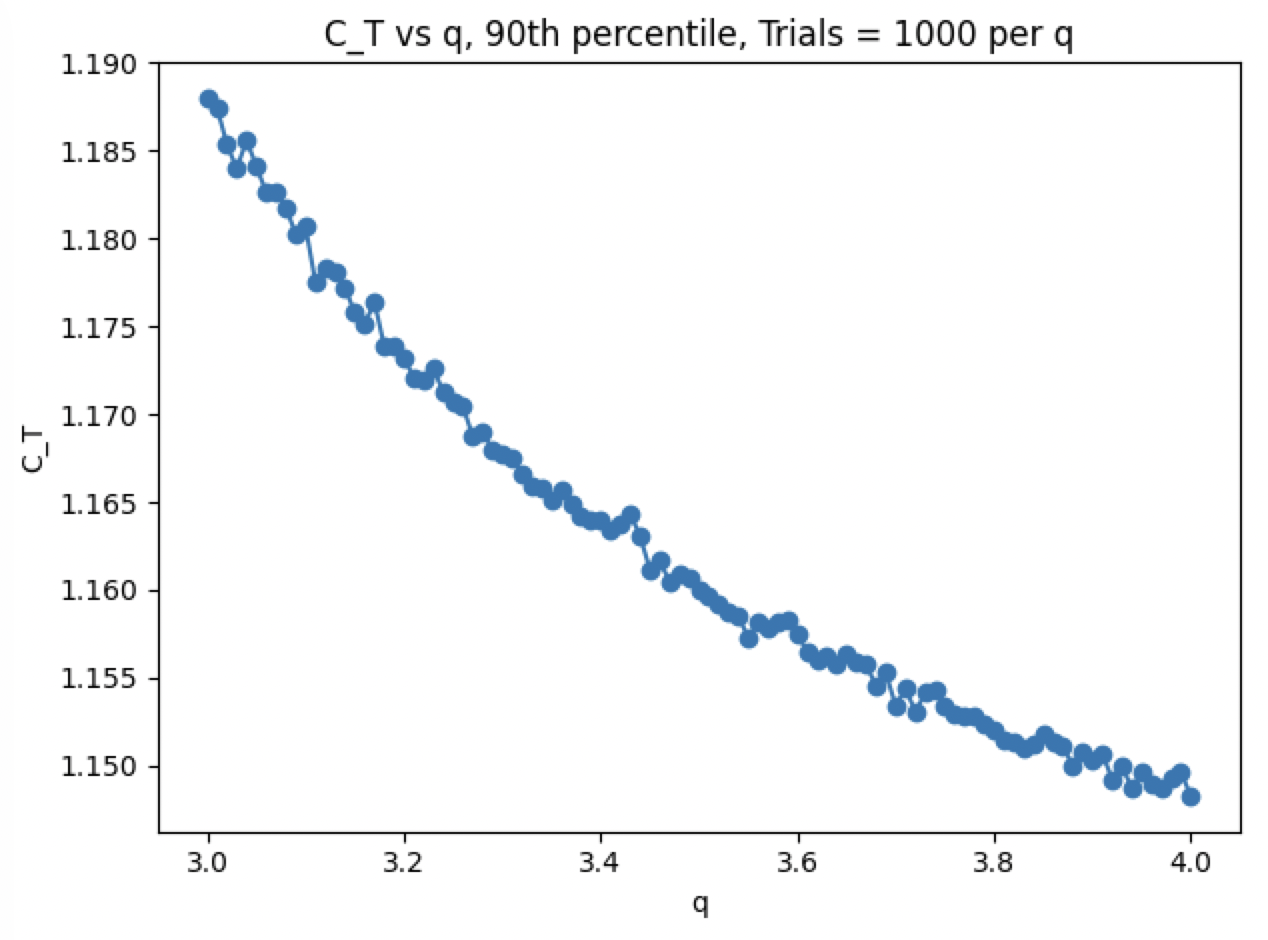}
    \caption{$C_T$ vs $q$}
    \label{fig:Ctvsq}
\end{figure}

In this fifth algorithm, we sought to see a relationship between $C_T$ and $N$. We fixed $q=4$ and let $N\in[10,100000]$.

\begin{figure}[H]
    \centering
    \includegraphics[width=0.5\linewidth]{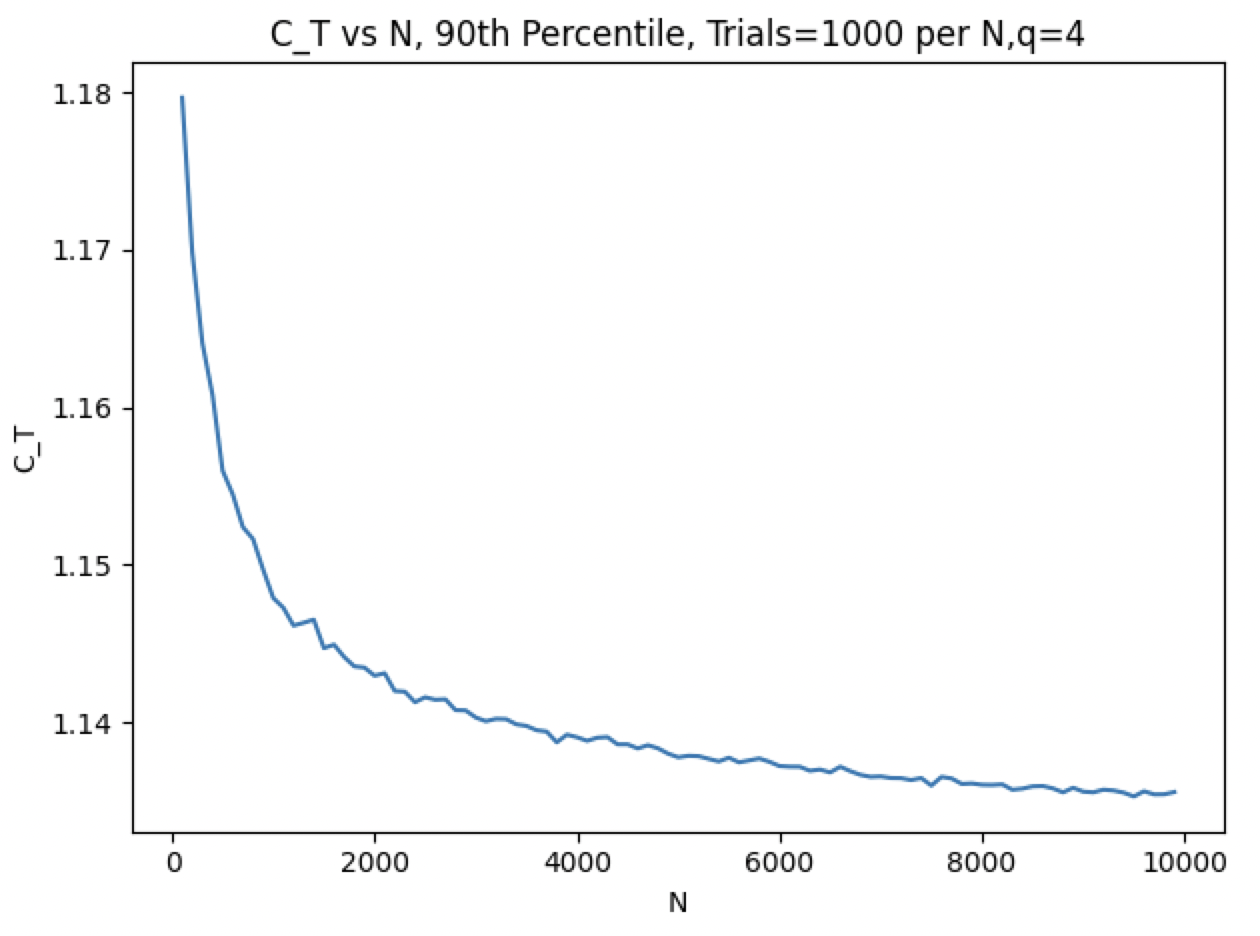}
    \caption{$C_T$ vs $N$}
    \label{fig:CtvsN}
\end{figure}

In this sixth algorithm, we created a heatmap for varying $q$ and $N$ and see how the two variables affect $C_T$. 

\begin{figure}[H]
    \centering
    \includegraphics[width=0.5\linewidth]{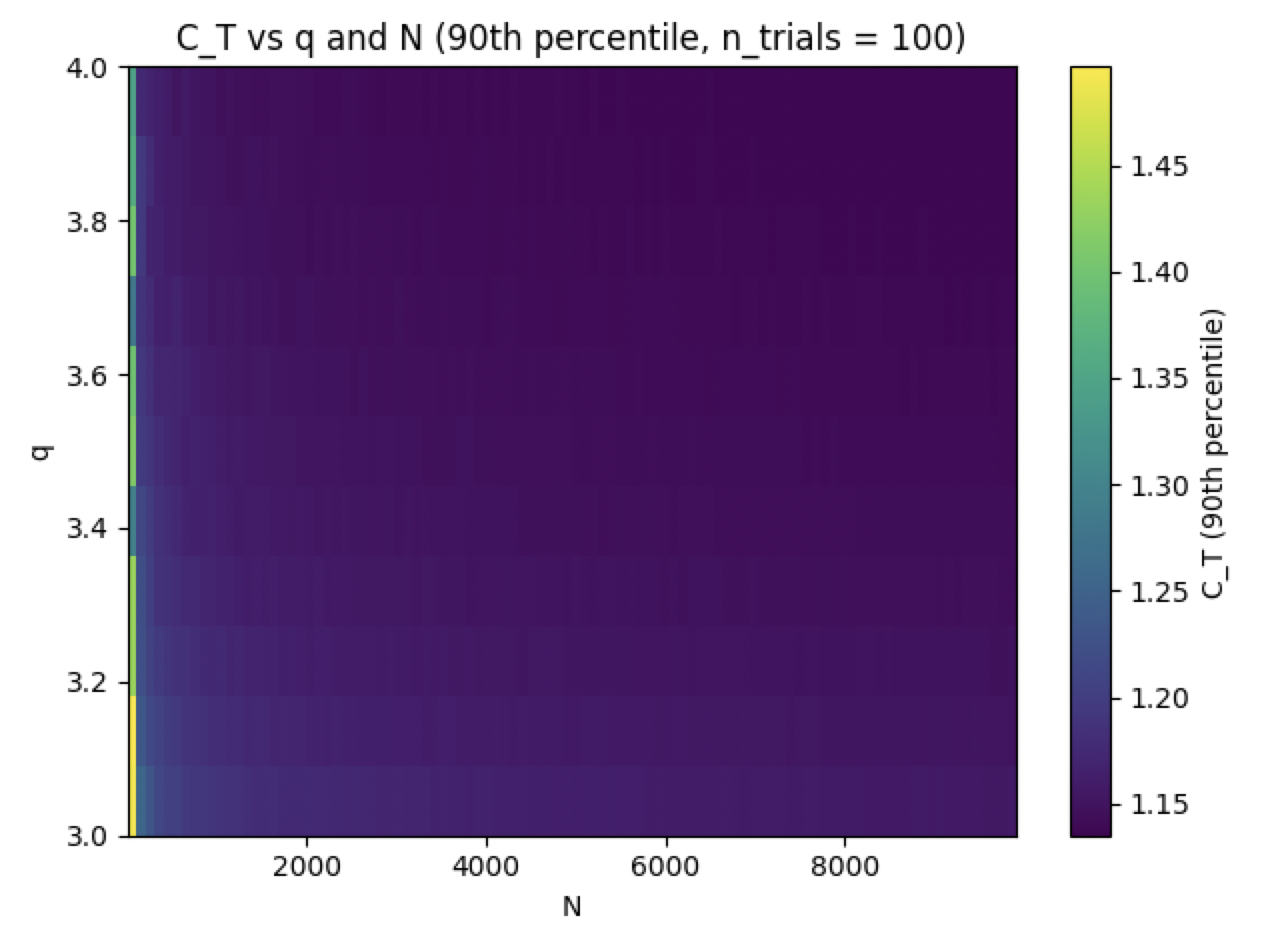}
    \caption{$C_T$ vs $q$ and $N$}
    \label{fig:CtvsqvsN}
\end{figure}

In this final algorithm, we have a graph comparing the values of $C(q)^{\frac{q}{q-2}}$ and $C_T$ varying both $q$ and $N$. 

\begin{figure}[H]
    \centering
\includegraphics[width=0.5\linewidth]{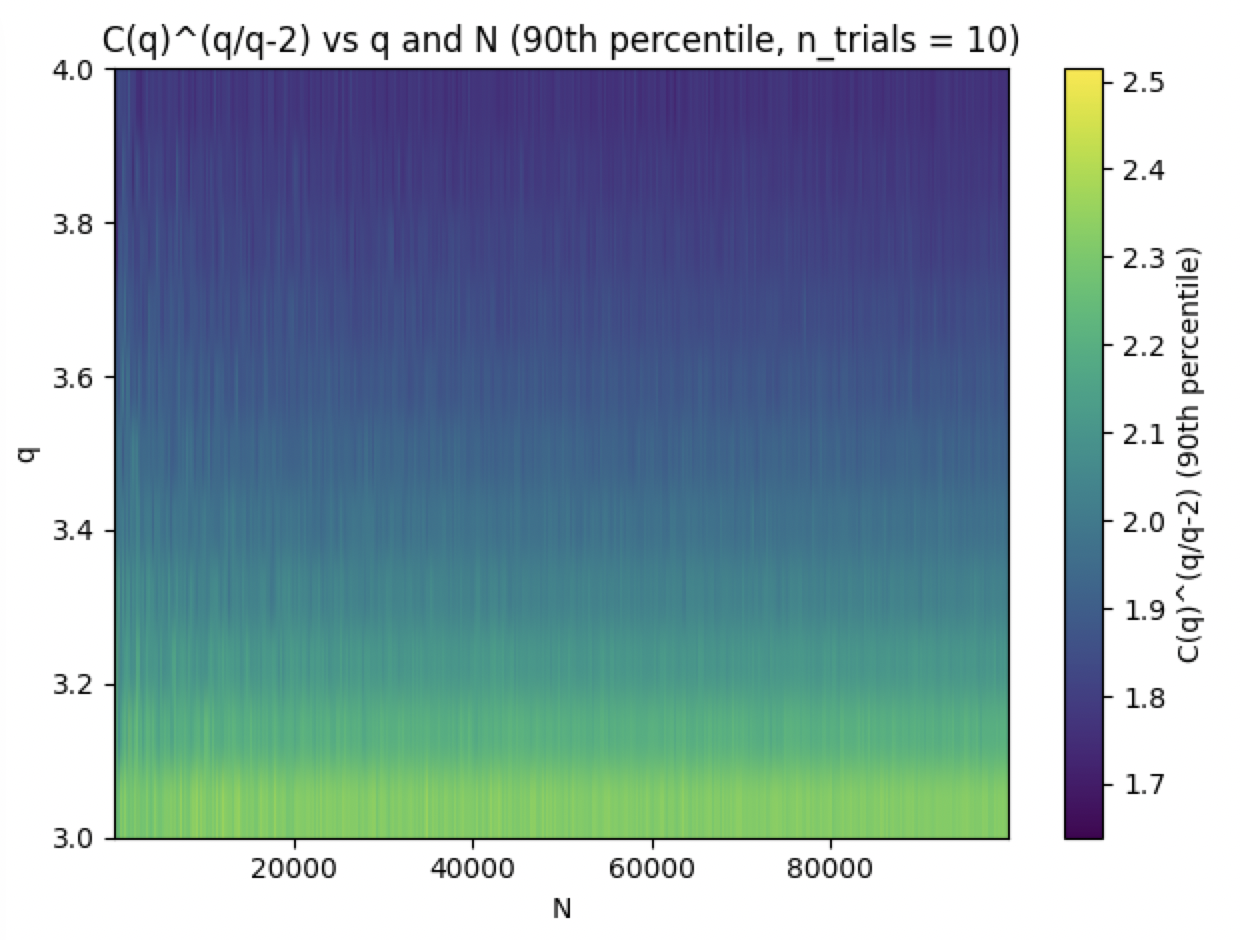}
    \caption{$(C(q))^\frac{q}{q-2}$ vs $q$ and $N$}
    \label{fig:C(q)vsqvsN}
\end{figure}

From the numerical experiments with Talagrand's constant and Bourgain's constant, we observe that 
\(C_T\le (C(q))^{\frac{q}{q-2}}.\) Numerically, we see that \(C_T\) is approximated by \(1-1.2\) whereas \((C(q))^{\frac{q}{q-2}}\) is approximated by \(1.5-2.5.\) This suggests that Talagrand's Constant behaves like a universal constant: the $N$ and $q$ dependence is negligible in our range of parameters, and its value can be reasonably approximated (up to a modest factor) by \((C(q))^{\frac{q}{q-2}}.\)

\section{Proofs of the main results}

\vskip.125in 

\subsection{Proof of Theorem \ref{theorem:FRuncertainty}}

 By the assumption, $$ \|f\|_{L^2(E^c)} \leq a \|f\|_2.$$
 It follows that 
$$ \|f\|_{L^2(E^c)} \leq a \left(\|f\|_{L^2(E)}+ \|f\|_{L^2(E^c)}\right).$$

 Combining we see that $$\|f\|_{L^2(E^c)} \leq \frac{a}{1-a} \|f\|_{L^2(E)}.$$

From this we conclude that $$ \|f\|_2\leq \|f\|_{L^2(E)}+ \|f\|_{L^2(E^c)} \leq \frac{1}{1-a} \|f\|_ {L^2(E)}.$$

The right hand side above is bounded by $$ \frac{1}{1-a} \|f\|_{\infty} \cdot {|E|}^{\frac{1}{2}},$$ which is bounded by $$ \frac{1}{1-a} \cdot N^{-\frac{1}{2}} \cdot {|E|}^{\frac{1}{2}} \cdot \|{\widehat{f}}\|_1.$$

It follows that $$ {(\text{FR}(f))}^2 \ge {(1-a)}^2 \cdot \frac{N}{|E|}.$$

On the other hand, by the assumption, $$\|{\widehat{f}}\|_{L^1(S^c)} \leq b \|{\widehat{f}}\|_1,$$ and  the same argument as above,  we have


$$\|{\widehat{f}}\|_1 \leq \frac{1}{1-b} \|{\widehat{f}}\|_{L^1(S)}.$$

By H\"older inequality and the Plancherel identity, we write: 
$$
\|{\widehat{f}}\|_{L^1(S)} \leq |S|^{1/2} \|\widehat f\|_2 = |S|^{1/2} \|f\|_{2}
$$

Combining this, it follows that $$ {(\text{FR}(f))}^2 \leq \frac{|S|}{{(1-b)}^2}.$$ 
 This completes the proof of the theorem.  

\vskip.125in 

\subsection{Proof of Theorem \ref{thm:relation-between-L1-L2-concentration}}

Throughout this proof, we will assume that $b$ is chosen such that $||\widehat{f}||_{L^2(S)} = (1-b^2) \cdot ||\widehat{f}||_{2}$, since if this fails, we can replace $b$ by a smaller quantity for which this equation holds. 
    \\ \\
    Since $f$ is an indicator function, note that $|\widehat{f}(m)| \leq \sqrt{N}$ for all $m$. (This is the only time we use the assumption that $f$ is an indicator function). Define the sets
    \[E_0 := \{m \in \mathbb{Z}_N: |\widehat{f}(m)| \leq 1 \}\]
    and for every $1 \leq k \leq \log_2(N)/2$,
    \[E_k := \{m\in \mathbb{Z}_N: 2^{k-1} < |\widehat{f}(m)| \leq 2^k\}.\]
    Since $||\widehat{f}||_2^2 = \sum_k \sum_{m\in E_k} |\widehat{f}(m)|^2$, we have
    \[\frac{1}{4} \cdot \sum_k |E_k|\cdot 2^{2k} \leq ||\widehat{f}||_2^2 \leq \sum_k |E_k|\cdot 2^{2k}.\]
    Similarly, $||f||_{L^2(S)}^2  = \sum_k \sum_{m\in E_k\cap S} |\widehat{f}(m)|^2$, so
    \[\frac{1}{4} \cdot \sum_k |E_k\cap S|\cdot 2^{2k} \leq ||\widehat{f}||_{L^2(S)}^2 \leq \sum_k |E_k\cap S|\cdot 2^{2k}.\]
    We will say an index $k$ is concentrated in the set $S$ if
    \[ ||\widehat{f}||^2_{L^2(E_k\cap S)} \geq (1-b^2) \cdot ||\widehat{f}||_{L^2(E_k)}^2.\]
    Since $\widehat{f}$ is $L^2$-concentrated in
    $S$ with accuracy $b$, and $S$ is the minimal set with this property, we can obtain the following lemma. 
    \begin{lemma} \label{lem:choosing-dyadic-scale-i}
    It is possible to choose an index $i$ which is concentrated in the set $S$ such that
    \begin{equation} \label{i1}
    ||\widehat{f}||^2_{L^2(E_i)} \geq \frac{1}{\log(N)}||\widehat{f}||_2^2,
    \end{equation}
    and
    \begin{equation}
        \label{eq-E_i-not-too-small}
        |E_i\cap S| \geq \frac{1}{\log N}\cdot |S|. 
    \end{equation}
    \end{lemma}
    \noindent The proof of the lemma is given after the completing the proof of the proposition.
    Let $i$ be the index given by the above lemma, then
    \begin{equation} \label{eq:dyadic-scale-i-dominates-L^2(S)}
       |E_i|\cdot2^{2i} \geq  \frac{1}{\log(N)} \cdot \frac{1}{1-b^2} \cdot ||\widehat{f}||_{L^2(S)}^2.
    \end{equation}
    Next we will consider $||\widehat{f}||_{1}$ and $||\widehat{f}||_{L^1(S)}$, which are estimated by:
    \[\frac{1}{2} \sum_{k} |E_k|\cdot 2^k \leq ||\widehat{f}||_1 \leq \sum_{k} |E_k|\cdot 2^k\]
    and
    \[\frac{1}{2} \sum_{k} |E_k\cap S|\cdot 2^k \leq ||\widehat{f}||_{L^1(S} \leq \sum_{k} |E_k\cap S|\cdot 2^k.\]
    Use the pigeonhole principle to choose two indices $j$ and $\ell$, where $0 \leq j,\ell \leq \log(N)/2$, as follows:
    \[|E_j\cap S|\cdot 2^j \geq \frac{1}{\log(N)/2}\cdot\sum_{k} |E_k\cap S|\cdot 2^k.\]
    and
    \[|E_\ell\cap S^C|\cdot 2^\ell \geq \frac{1}{\log(N)/2}\cdot \sum_{k} |E_k\cap S^C|\cdot 2^k.\]
    Rearranging the above two estimates gives:
    \begin{equation}\label{eq-j-L^1-S-estimate}
    \frac{1}{2} |E_j\cap S| \cdot 2^j  \leq ||\widehat{f}||_{L^1(S)} \leq \frac{\log(N)}{2} \cdot |E_j\cap S| \cdot 2^j\end{equation}
    and
    \begin{equation}\label{eq-j-l-L^1-estimate}
    \frac{1}{2} \left(|E_j\cap S| \cdot 2^j + |E_\ell\cap S^C| \cdot 2^\ell \right) \leq ||\widehat{f}||_1 \leq \frac{\log(N)}{2} \cdot \left(|E_j\cap S| \cdot 2^j + |E_\ell\cap S^C| \cdot 2^\ell \right).\end{equation}
    We are now in a position to compare the ratio $||\widehat{f}||_{L^1(S)}/||\widehat{f}||_{L^2(S)}$ with $\text{FR}(f)$. On the one hand, using $(\ref{i1})$ and $(\ref{eq-j-l-L^1-estimate})$,
    \begin{equation} \label{FR-dyadic-estimate}
    \frac{1}{\sqrt{\log(N)}} \cdot\frac{|E_j\cap S| \cdot 2^j + |E_\ell\cap S^C| \cdot 2^\ell}{|E_i|^{1/2}\cdot 2^{i}} \leq \text{FR}(f) \leq \frac{\log(N)}{2} \cdot\frac{|E_j\cap S| \cdot 2^j + |E_\ell\cap S^C| \cdot 2^\ell}{|E_i|^{1/2}\cdot 2^{i}}
    \end{equation}
    On the other hand, using $(\ref{eq:dyadic-scale-i-dominates-L^2(S)})$ and $(\ref{eq-j-L^1-S-estimate})$,
    \begin{equation} \label{FR-on-S-dyadic-estimate}
    \frac{1}{\sqrt{\log(N)}}\cdot \frac{1}{\sqrt{1-b^2}} \cdot\frac{|E_j\cap S| \cdot 2^j}{|E_i|^{1/2}\cdot 2^{i}} \leq \frac{||\widehat{f}||_{L^1(S)}}{||\widehat{f}||_{L^2(S)}} \leq \frac{\log(N)}{2} \cdot \frac{1}{\sqrt{1-b^2}} \cdot\frac{|E_j\cap S| \cdot 2^j}{|E_i|^{1/2}\cdot 2^{i}}
    \end{equation}
    Using $(\ref{FR-dyadic-estimate})$ and $(\ref{FR-on-S-dyadic-estimate})$, we can write
    \begin{equation}
        \label{}
        \frac{||\widehat{f}||_{L^1(S)}}{||\widehat{f}||_{L^2(S)}} \geq \frac{2}{(\log(N))^{3/2}} \cdot \frac{1}{\sqrt{1-b^2}}\cdot \frac{|E_j\cap S| \cdot 2^j}{|E_j\cap S| \cdot 2^j + |E_\ell\cap S^C| \cdot 2^\ell} \cdot \text{FR}(f)
    \end{equation}
    To complete the proof, we need to show that
    \[
    \frac{|E_j\cap S| \cdot 2^j}{|E_j\cap S| \cdot 2^j + |E_\ell\cap S^C| \cdot 2^\ell} \geq \frac{1}{1 + (\log N)^{3/2}\cdot\frac{b \cdot |S^C|^{1/2}}{\sqrt{1-b^2} \cdot |S|^{1/2}}},
    \]
    which is equivalent to showing that
    \begin{equation} \label{eq-theta-expression}
    \frac{|E_\ell\cap S^C|\cdot 2^\ell}{|E_j\cap S|\cdot 2^j} \leq (\log N)^{3/2} \cdot \frac{b\cdot |S^C|^{1/2}}{\sqrt{1-b^2}\cdot |S|^{1/2}}.
    \end{equation}
    To show this, first apply H\"older's inequality:
    \[||\widehat{f}||_{L^1(S^C)} \leq |S^C|^{1/2} \cdot ||\widehat{f}||_{L^2(S^C)} = |S^C|^{1/2} \cdot b \cdot ||\widehat{f}||_2.\]
    Using $(\ref{eq:dyadic-scale-i-dominates-L^2(S)})$ with this estimate, we get
    \begin{equation} \label{eq-El-holder}
    |E_\ell\cap S^C|\cdot 2^\ell \leq |S^C|^{1/2} \cdot b \cdot  \left(\log(N) \cdot |E_i|^{1/2}\cdot 2^i \right).
    \end{equation}
    Next we will show that
    \begin{equation} \label{eq-Ej-reverse-holder}
    |E_j\cap S|\cdot 2^j \geq \frac{1}{(\log N)^{1/2}} \cdot|S|^{1/2} \cdot  \sqrt{1-b^2}\cdot\left( |E_i|^{1/2}\cdot 2^i \right).
    \end{equation}
    After $(\ref{eq-Ej-reverse-holder})$ is shown, combining it with $(\ref{eq-El-holder})$ gives $(\ref{eq-theta-expression})$. To show $(\ref{eq-Ej-reverse-holder})$, we first observe that
    \begin{equation}
    \label{eq-j>i-in-L^1}
    |E_j\cap S|\cdot 2^j \geq |E_i\cap S|\cdot 2^i,
    \end{equation}
    because if this is not true, we could replace $j$ with $i$. This observation, combined with $(\ref{eq-E_i-not-too-small})$, leads to the following inequality: 
    \[|E_j\cap S|\cdot 2^j \geq |E_i\cap S|^{1/2} \cdot |E_i\cap S|^{1/2}\cdot 2^i \geq \frac{|S|^{1/2}}{(\log N)^{1/2}} \cdot \sqrt{1-b^2} \cdot |E_i|^{1/2}\cdot 2^i.\]
    This gives $(\ref{eq-Ej-reverse-holder})$, which completes the proof of the proposition, except for the proof of Lemma 1 which is below.

\begin{proof}[Proof of Lemma \ref{lem:choosing-dyadic-scale-i}]
        Let $I$ be the set of indices $k$, where $0 \leq k \leq \log(N)/2$, which are concentrated in the set $S$. Define two sets $B_1$ and $B_2$ by:
        \[B_1 = \bigcup_{k\in I} E_k \text{ and } B_2 = \bigcup_{k\in I^C} E_k.\]
        We first claim that
        \begin{equation}
            \label{eq-scales-concentrated in-S-dominate}
            ||\widehat{f}||^2_{L^2(B_1)} \geq \frac{1}{2} \cdot ||\widehat{f}||_{L^2}^2.
        \end{equation}
        Indeed, if $(\ref{eq-scales-concentrated in-S-dominate})$ does not hold, then there exists a number $\varepsilon_1 > 0$ such that
        \[||\widehat{f}||^2_{L^2(B_1)} = \left(\frac{1}{2} -\varepsilon_1\right)\cdot ||\widehat{f}|^2_2 \text{ and } ||\widehat{f}||^2_{L^2(B_2)} = \left(\frac{1}{2} +\varepsilon_1\right)\cdot ||\widehat{f}|^2_2.\]
        From the definition of $I$, we know that $\widehat{f}\cdot1_{B_1}$ is concentrated in $S$ with accuracy (less than) $b$, thus, there exists $\varepsilon_2 > 0$ such that
        \[||\widehat{f}||^2_{L^2(B_1\cap S^C)} = (b^2 - \varepsilon_2) \cdot ||\widehat{f}|^2_{L^2(B_1)} 
        \text{ and } 
        ||\widehat{f}||^2_{L^2(B_2\cap S^C)} = (b^2 + \varepsilon_2) \cdot ||\widehat{f}|^2_{L^2(B_2)}.\]
        Consequently,
        \begin{align*}
           ||\widehat{f}||_{L^2(S^C)}^2 
            & =  ||\widehat{f}||^2_{L^2(B_1\cap S^C)} +||\widehat{f}||^2_{L^2(B_2\cap S^C)} \\
            &= (b^2 - \varepsilon_2) \cdot ||\widehat{f}|^2_{L^2(B_1)} + (b^2 + \varepsilon_2) \cdot ||\widehat{f}|^2_{L^2(B_2)} \\
            &= (b^2 - \varepsilon_2) \cdot \left(\frac{1}{2} -\varepsilon_1\right)\cdot ||\widehat{f}|^2_2 + (b^2 + \varepsilon_2) \cdot \left(\frac{1}{2} +\varepsilon_1\right)\cdot ||\widehat{f}|^2_2 \\
            & > b^2\cdot ||\widehat{f}||^2_2,
        \end{align*}
        which is impossible, since $f$ is concentrated in $S$ with accuracy $b$. This proves $(\ref{eq-scales-concentrated in-S-dominate})$.
        We now apply the pigeonhole principle to $(\ref{eq-scales-concentrated in-S-dominate})$ to see that there exists some $i \in I$ such that
        \[
        ||\widehat{f}||^2_{L^2(E_i)} \geq \frac{1}{|I|} \sum_{k
        \in I} ||\widehat{f}||^2_{L^2(E_k)} = \frac{1}{|I|}||\widehat{f}||^2_{L^2(B_1)} \geq \frac{1}{2|I|} \cdot ||\widehat{f}||_{L^2}^2.
        \]
        Let $J$ be the set of all indices in $k\in I$ for which above inequality is true. Since the above inequality implies $(\ref{i1})$, it remains to show that there is some index $i \in J$ which also satisfies $(\ref{eq-E_i-not-too-small})$. 
        To see this, suppose there exists no such index in $J$. Then, for every $k \in J$, we have
        \[|E_k \cap S| < \frac{|S|}{\log(N)} \;\;\; \text{ and thus } \;\;\;  \frac{||\widehat{f}||_{L^2(E_k \cap S)}^2}{|E_k\cap S|} > \frac{\log(N)}{2|I|\cdot|S|} \cdot (1-b^2)\cdot||\widehat{f}||_2^2.\]
        Define $B_3 = \bigcup_{k\in J} E_k$. Then
        \begin{align*}
            ||\widehat{f}||_{L^2(B_3 \cap S)}^2 &= \sum_{k\in J} ||\widehat{f}||_{L^2(E_k \cap S)}^2 \\
            &> \frac{\log(N)}{2|I|\cdot|S|} \cdot (1-b^2)\cdot||\widehat{f}||_2^2\cdot\sum_{k\in J} |E_j\cap S| .\\
            &\geq \frac{\log(N)}{2|I|} \cdot (1-b^2)\cdot||\widehat{f}||_2^2\cdot \frac{|B_3\cap S|}{|S|} .\\
            &\geq (1-b^2)\cdot ||\widehat{f}||_2^2\cdot \frac{|B_3\cap S|}{|S|}.
        \end{align*}
        This contradicts the hypothesis in the theorem statement that $S$ is minimal with respect to the $L^2$-concentration property.
        This completes the proof of the lemma.
        \end{proof}

\subsection{Proof of Theorem \ref{thm:concentration}} We have 
$$ {\|\widehat{f}\|}_2 \leq {\|\widehat{1_Mf}\|}_2+{\|\widehat{1_{M^c}f}\|}_2.$$

By Talagrand, 
\begin{align}\notag {\|\widehat{1_Mf}\|}_2 &=N^{\frac{1}{2}} {\|\widehat{1_Mf}\|}_{L^2(\mu)}
\\\notag
&\leq C_T \sqrt{N} \sqrt{\log(N) \log \log(N)} {\|\widehat{1_Mf}\|}_{L^1(\mu)} 
\\\notag
& = C_T N^{-\frac{1}{2}} \sqrt{\log(N) \log \log(N)} {\|\widehat{1_Mf}\|}_1. 
\end{align}

By the concentration assumption, 
$${\|\widehat{1_{M^c}f}\|}_2 \leq r {\|f\|}_2. $$

It follows that 
$$ {\|\widehat{f}\|}_2 (1-r) \leq C_T N^{-\frac{1}{2}} \sqrt{\log(N) \log \log(N)} {\|\widehat{1_Mf}\|}_1, $$ and we conclude that 
$$ {\|\widehat{f}\|}_2 \leq \frac{C_T N^{-\frac{1}{2}} \sqrt{\log(N) \log \log(N)}}{1-r}{\|\widehat{1_Mf}\|}_1.$$

We must now unravel ${\|\widehat{1_Mf}\|}_1$. This expression equals 
\begin{align*}
    {\|\widehat{f}-\widehat{1_{M^c}f}\|}_1 &\leq {\|\widehat{f}\|}_1+{\|\widehat{1_{M^c}}f\|}_1 \\
    &\leq {\|\widehat{f}\|}_1 +rN^{\frac{1}{2}} {\|\widehat{f}\|}_2.
\end{align*}

It follows that 
$$ {\|\widehat{f}\|}_2 \cdot \left(1-\frac{C_T r \sqrt{\log(N) \log \log(N)}}{1-r} \right) \leq \frac{C_T N^{-\frac{1}{2}} \sqrt{\log(N) \log \log(N)}}{1-r} {\|\widehat{f}\|}_1.$$

We conclude that 
$$ \frac{{\|\widehat{f}\|}_{L^1(\mu)}}{{\|\widehat{f}\|}_{L^2(\mu)}} \ge \frac{\left(1-r\frac{C_T \sqrt{\log(N) \log \log(N)}}{1-r} \right)}{\frac{C_T \sqrt{\log(N) \log \log(N)}}{1-r}}, $$ as claimed. \qed

\vskip.125in 

\subsection{Proof of Theorem \ref{thm:Linftypolynomialapprox}} 
    The proof is probabilistic, so suppose $f:\mathbb Z_N\to \mathbb C$ and define the random function $Z:\mathbb Z_N\to \mathbb C$, where for each $m$,
    $$Z(x)=\|\widehat f\|_1\operatorname{sgn}(\widehat{f}(m))N^{-\frac{1}{2}}\chi(mx)$$
    with probability $\frac{|\widehat{f}(m)|}{\|\widehat{f}\|_1}$, where $\operatorname{sgn}(z) = \frac{z}{|z|}$. Observe that the expected value of $Z$ is just our function $f$.
    
    Let $Z_1,\dots, Z_k$ be i.i.d. random functions with the same distribution as $Z$. Note that for each $i$, we have
    $$|Z_i(x)| = N^{-\frac{1}{2}} \|\widehat f\|_1,$$
    and thus
    \begin{gather}
        -N^{-\frac12}\|\widehat f\|_1 \leq \operatorname{Re}(Z_i(x)) \leq N^{-\frac12}\|\widehat f\|_1, \label{eq:linftyapprox realZbound}\\
        -N^{-\frac12}\|\widehat f\|_1 \leq \operatorname{Im}(Z_i(x)) \leq N^{-\frac12}\|\widehat f\|_1. \label{eq:linftyapprox imZbound}
    \end{gather}
    Next, by the union bound note that for each $x$ we have
    \begin{align*}
        \mathbb{P}\left(\left|\mathbb{E}\left(Z(x)\right)-\frac{1}{k}\sum_{i=1}^{k}Z_{i}(x)\right|\geq \eta\right) \leq
        \mathbb{P}\left(\left|\mathbb{E}\left(\operatorname{Re}(Z(x))\right)-\frac{1}{k}\sum_{i=1}^{k}\operatorname{Re}(Z_{i}(x))\right|\geq \frac{\eta}{2}\right) \\
        +\mathbb{P}\left(\left|\mathbb{E}\left(\operatorname{Im}(Z(x))\right)-\frac{1}{k}\sum_{i=1}^{k}\operatorname{Im}(Z_{i}(x))\right|\geq \frac{\eta}{2}\right).
    \end{align*}


    
    By Hoeffding's inequality (\cite{Hoeffding63}), as well as \eqref{eq:linftyapprox realZbound} and \eqref{eq:linftyapprox imZbound}, we get that the right-hand side is bounded by
    \begin{equation} \label{eq:hoeffdingsLinftybound}
        4 \exp \left( - \frac{2 \left( \frac{\eta k}{2} \right)^2}{k \left( 2N^{-\frac12}\|\widehat f\|_1 \right)^2} \right) = 4 \exp \left( - \frac{\eta^2 k N}{8\|\widehat f\|_1^2} \right).
    \end{equation}
    Now, suppose the right-hand side in \eqref{eq:hoeffdingsLinftybound} is less than $\frac 1 N$, and set $P(x) = \frac{1}{k} \sum_{i=1}^k Z_i(x)$. Then by the union bound as well as the earlier observation that $\mathbb E[Z(x)] = f(x)$, we have
    $$\mathbb P \left( \|f-P\|_\infty \geq \eta \right) \leq \sum_{x\in\mathbb Z_N} \mathbb P\left( |f(x) - P(x)| \geq \eta \right) < 1,$$
    so there is a deterministic choice of $Z_1,\dots,Z_k$ with $\|f-P\|_\infty < \eta$.
    The assumption that \eqref{eq:hoeffdingsLinftybound} is less than $\frac1N$ amounts to
    $$4N \exp \left( - \frac{\eta^2 k N}{8\|\widehat f\|_1^2} \right) < 1,$$
    i.e.,
    $$\log(4N) - \frac{\eta^2 k N}{8\|\widehat f\|_1^2}<0. $$
    Then by setting $\eta$ to equal $\varepsilon\|f\|_{\infty}$ and rearranging, we get that for any $k$ such that
    $$
    k>8\left(\frac{\|\widehat{f}\|_{L^{1}(\mu)}}{\|f\|_{L^{\infty}}}\right)^{2}\frac{N\log(4N)}{\varepsilon^{2}},
    $$
    there is a $P$ such that
    $$\|f - P\|_\infty < \varepsilon \|f\|_\infty,$$
    and we are done. \qed

\vskip.25in 

\subsection{Proof of Theorem \ref{thm:L2polynomialapprox}}
    We proceed as in the proof of Theorem \ref{thm:Linftypolynomialapprox}, so suppose $f:\mathbb Z_N\to \mathbb C$, and again define the random function $Z:\mathbb Z_N\to \mathbb C$, where for each $m$,
    $$Z(x) = \|\widehat f\|_1 \operatorname{sgn}(\widehat f(m)) N^{-\frac12} \chi(mx)$$
    with probability $\frac{|\widehat f(m)|}{\|\widehat f\|_1}$. We again have that $\mathbb E[Z(x)] = f(x)$. Moreover, note that for each $x$,
    \begin{align*}
        \mathbb E|Z(x)|^2 &= \sum_{m\in \mathbb Z_N} \left| \|\widehat f\|_1 \frac{\widehat f(m)}{|\widehat f(m)|} N^{-\frac12} \chi(mx) \right|^2 \cdot \frac{|\widehat f(m)|}{\|\widehat f\|_1} \\
        &= \frac1N \|\widehat f\|_1\sum_{m\in \mathbb Z_N} |\widehat f(m)| \\
        &= \frac1N \|\widehat f\|_1^2,
    \end{align*}
    and thus the variance of $Z(x)$ is
    $$\operatorname{Var}(Z(x)) = \mathbb E|Z(x)|^2 - |\mathbb E[Z(x)]|^2 = \frac1N \|\widehat f\|_1^2 - |f(x)|^2.$$
    Now, let $Z_1,\dots,Z_k$ be random i.i.d. functions with distribution $Z$, and define the random trigonometric polynomial $P$ by
    $$P(x) = \frac{1}{k} \sum_{i=1}^k Z_i(x).$$
    Note that $\mathbb E[P(x)] = f(x)$ and that by independence,
    $$\operatorname{Var}(P(x)) = \frac{1}{k}\operatorname{Var}(Z(x)),$$
    and thus
    \begin{align*}
        \mathbb E\|f - P\|_2^2 &= \sum_{x\in\mathbb Z_N} \mathbb E|f(x) - P(x)|^2 \\
        &= \sum_{x\in\mathbb Z_N} \operatorname{Var}(P(x)) \\
        &= \frac{1}{k} \sum_{x\in \mathbb Z_N} \operatorname{Var}(Z(x)) \\
        &= \frac{1}{k} \sum_{x\in \mathbb Z_N} \frac1N \|\widehat f\|_1^2 - |f(x)|^2 \\
        &= \frac{1}{k}\left(\|\widehat f\|_1^2 - \|f\|_2^2\right).
    \end{align*}
    Now, if we assume this final value is less than $\eta^2 \|f\|_2^2$, then there exists a deterministic choice of $P$ such that $\|f - P\|_2 < \eta \|f\|_2$. \\
    This assumption on $k$ amounts to
    \begin{align*}
        k &> \frac{1}{\eta^2} \cdot \frac{\|\widehat f\|_1^2 - \|f\|_2^2}{\|f\|_2^2} \\
        &= \frac{1}{\eta^2} \left( \frac{\|\widehat f\|_1^2}{\|f\|_2^2} - 1 \right) \\
        &= \frac{\text{FR}(f)^2 - 1}{\eta^2},
    \end{align*}
    and thus for any such $k$, there is a trigonometric polynomial $P$ with $\|f-P\|_2 < \eta\|f\|_2$, and we are done. \qed

\subsection{Proof of Theorem \ref{thm:L1polynomialapprox}} 

Again define the random function $Z:\mathbb Z_N\to\mathbb C$ by choosing $m\in\mathbb Z_N$ with probability $\frac{|\widehat f (m)|}{\|\widehat f\|_1}$ and setting
$$Z(x) = \|\widehat f\|_1 \operatorname{sgn}(\widehat f (m)) N^{-1/2} \chi(mx).$$
If $Z_1,\dots,Z_k$ are independent copies of $Z$, and if
$$P(x) = \frac{1}{k} \sum_{i=1}^k Z_i(x)$$
Then $\mathbb E[P(x)] = f(x)$, and moreover by the bound \eqref{eq:hoeffdingsLinftybound} from the proof of Theorem \ref{thm:Linftypolynomialapprox}, we get that for every $x\in\mathbb Z_N$,
$$\mathbb P(|f(x) - P(x)| \geq t) \leq 4\exp\left( -\frac{t^2kN}{8\|\widehat f\|_1^2} \right).$$
Consequently, we have that
\begin{align*}
    \mathbb E\|f-P\|_1 &= \sum_{x\in\mathbb Z_N} \mathbb E|f(x) - P(x)| \\
    &= \sum_{x\in\mathbb Z_N} \int_0^\infty \mathbb P(|f(x)-P(x)| \ge t)dt \\
    &\leq \sum_{x\in\mathbb Z_N} \int_0^\infty 4 \exp\left( -\frac{t^2kN}{8\|\widehat f\|_1^2} \right)dt \\
    &= 4N \int_0^\infty \exp\left( -t^2 \cdot \frac{kN}{8\|\widehat f\|_1^2} \right)dt \\
    &= 4N\left(\frac{\sqrt 8\|\widehat f\|_1}{\sqrt{kN}}\right) \int_0^\infty \exp\left(-t^2\right)dt \\
    &= \frac{2\sqrt{8\pi N}}{\sqrt k}\|\widehat f\|_1.
\end{align*}
Now, if we assume this final value is less than $\eta\|f\|_1$, then there is a deterministic choice of $P$ such that $\|f-P\|_1 < \eta \|f\|_1$. \\
This assumption amounts to
$$\frac{2\sqrt{8\pi N}}{\sqrt k}\|\widehat f\|_1 < \eta\|f\|_1,$$
or in other words
$$k > 32\pi \left( \frac{\|\widehat f\|_1}{\|f\|_1} \right)^2 \frac{N}{\eta^2}.$$
Thus for any such $k$, there is a trigonometric polynomial $P$ with $\|f-P\|_1 < \eta\|f\|_1$, which completes the proof. \qed

\vskip.25in 

\subsection{Proof of Theorem \ref{thm:FourierRatio and AlgorithmicRateDistortion}}

Let $f:\mathbb{Z}\to\mathbb{C}$ and let $\varepsilon>0$. If $k = \frac{\text{FR}(f)^2}{\varepsilon^2}$, then by Theorem \ref{thm:L2polynomialapprox}, there exists a trigonometric polynomial $P$ of the form
$$P(x) = \sum_{i=1}^k c_i \chi(m_ix)$$
such that $\|f-P\|_2 < \varepsilon \|f\|_2$. We show that there is another trigonometric polynomial $P'$ with low Kolmogorov complexity such that $\|P-P'\|_2 < \varepsilon \|f\|_2$, since this will imply $r_f(2\varepsilon) \leq \K(P')$.

First, note that
\begin{align*}
    \widehat P(m_j) &= \frac{1}{\sqrt N} \sum_{x\in\mathbb Z_N} \chi(-xm_j)P(x) \\
    &= \frac{1}{\sqrt N} \sum_{x\in \mathbb Z_N} \sum_{i=1}^k c_i \chi(x(m_i-m_j)) \\
    &= \sqrt N c_j.
\end{align*}
Now, let $M$ be a constant to be chosen later, and define $P'$ to be
$$P'(x) = \sum_{i=1}^k c_i' \chi(m_ix),$$
where $c_i'$ is $c_i$, truncated to $M$ digits after the decimal point. 
Note that since
\begin{align*}
    |c_j'| &\leq |c_j| \\
    &= \frac{1}{\sqrt N} \left| \widehat P (m_j) \right | \\
    &\leq \frac1{\sqrt N} \|\widehat P\|_2 \\
    &= \frac{1}{\sqrt N} \| P \|_2 \\
    &\leq \frac{(1+\varepsilon)\|f\|_2}{\sqrt N},
\end{align*}
and since each $c_j'$ contains at most $M$ digits after the decimal point, we can encode each $c_j'$ in length
$$C_U\left( M + \log\left(\frac{(1+\varepsilon)\|f\|_2}{\sqrt N}\right) \right),$$
where $C_U$ is some constant depending on the fixed universal Turing machine.

Since encoding each $m_i$ takes length $C_U\log N$, we see that we can encode $P'$ using a Turing machine of length
$$C_Uk\left( M + \log\left(\frac{(1+\varepsilon)\|f\|_2}{\sqrt N}\right) + \log N \right) + C_U',$$
where $C_U'$ is another constant depending on the universal Turing machine, giving the length required to obtain the trigonometric polynomial $P'$ from the values of $c_i'$ and $m_i$. This thus gives a bound on $\K(P')$.

Next, note that since $c_i'$ is obtained by truncating $c_i$, we have that
$$|c_i - c_i'| \leq 2^{-M},$$
so that 
\begin{align*}
    \|P-P'\|_2^2 &= \|\widehat P - \widehat {P'}\|_2^2 \\
    &= N\sum_{i=1}^k |c_i - c_i'|^2 \\
    &\leq Nk2^{-2M}. 
\end{align*}
Thus, if we want $\|P - P'\|_2^2 \leq \varepsilon^2\|f\|_2^2$, we need that
$$k2^{-2M} N \leq \varepsilon^2\|f\|_2^2,$$
so in other words that
$$M \geq \frac12 \log \left ( \frac{k N}{\varepsilon^2\|f\|_2^2} \right) = \log \left( \frac{\sqrt k \sqrt N}{\varepsilon \|f\|_2} \right).$$
 
Taking $M$ such that equality holds above, we have that
\begin{align*}
    \K(P') &\leq C_Uk\left( M + \log\left(\frac{(1+\varepsilon)\|f\|_2}{\sqrt N}\right) + \log N \right) + C_U' \\
    &= C_U k\left( \log \left( \frac{\sqrt k \sqrt N}{\varepsilon \|f\|_2} \right) + \log\left(\frac{(1+\varepsilon)\|f\|_2}{\sqrt N}\right) + \log N \right) + C_U' \\
    &= C_Uk \log\left( \frac{(1+\varepsilon)N\sqrt k}{\varepsilon} \right) + C_U'.
\end{align*}
and thus since
$$\|f - P'\|_2 \leq \|f - P\|_2 + \|P-P'\|_2 < 2\varepsilon\|f\|_2,$$
we get that
\begin{align*}
    r_f(2\varepsilon) &\leq C_Uk \log\left( \frac{(1+\varepsilon)N\sqrt k}{\varepsilon} \right) + C_U' \\
    &\leq O \left( k \log\left( \frac{(1+\varepsilon)N\sqrt k}{\varepsilon} \right) \right).
\end{align*} \qed

\vskip.125in    
\subsection{Proof of Theorem \ref{theorem:alg}} 

We shall need the following classical definition. See e.g. \cite{CT05}, \cite{C08}, and \cite{Rudelson}. 

\begin{definition}[Restricted Isometry Condition]
A matrix $A \in \mathbb{R}^{m \times n}$ is said to satisfy the 
\textit{restricted isometry condition of order $k$} with constant 
$\delta_k \in (0,1)$ if for all vectors $x \in \mathbb{R}^n$ with at 
most $k$ nonzero entries and for sufficient large $C$, the following holds:
\[
C(1 - \delta_k)\|x\|_2^2 \leq \|Ax\|_2^2 \leq C(1 + \delta_k)\|x\|_2^2.
\]

\end{definition}
\vskip.125in 

Define $\Phi := \left (\frac{\chi(s \cdot r)}{\sqrt{N}} \right )_{s,r} \in \mathbb{C}^{N \times N}$ to be the $N \times N$ inverse Fourier matrix.  Set $S=\frac{r^2}{\varepsilon^2}$.  Consider $\delta_{4S} = 1/4$, a restricted isometry constant corresponding to vectors with at most $4S$ non-zero entries.  Noting that $\delta_{3S} \leq \delta_{4S}$.  Sample $q' = CS \log^2(S) \log(N)$ rows of $\Phi$ uniformly and independently, denoting $q \leq q'$ to be the number of distinct rows.  Denote the resulting matrix of distinct rows as $A \in \mathbb{C}^{q \times N}$.   Then, with probability $1-2^{-\Omega( (\log N) \cdot  (\log S))}$, $A$ satisfies the restricted isometry condition with  $\delta_{4S} =1/4$, by Theorem 4.5 of \cite{Haviv15}, and with high-probability the condition
\begin{equation}
\label{eq:rip_1}
\delta_{3S}+ 3 \delta_{4S} < 2 
\end{equation}
is satisfied.

\vspace{0.125in}
Let $f \in \mathcal{C}(r)$.  By Theorem \ref{thm:L2polynomialapprox}, there exists an $S$-sparse polynomial $P$, such that,
$$
\|f - P\|_2 \leq \|f\|_2 \cdot \varepsilon. 
$$
Let $f_X$ be the restriction of $f$ on $X$, $y=f_X,x_0 =\widehat{P}$. We get $\|y-Ax_0\| = \|f_X-A\widehat{P}\|_2 =\|f-P\|_{L^2(X)}\leq \|f-P\|_2 \leq  \|f\|_2 \cdot \varepsilon$.
Consider the following optimization problem:
$$
\text{ (*): } \min\|\widehat{x}\|_1 \text{ subject to } \|f-x\|_{L^2(X)} \leq \eta.
$$

The following result (Theorem 1 of \cite{Candes05}) gives us guarantees for (*).
\begin{theorem}
\label{TaoTheorem}
Let $S$ be such that $\delta_{3S} + 3\delta_{4S} < 2$. Then for any signal $x_0$ supported on $T_0$
with $ |T_0| \leq S$ and any perturbation $e$ with $\|e\|_2 \leq \eta$, the solution $x^{*}$
to (*) obeys
$$
\|\widehat{x^{*}} -x_0\|_2 \leq C_S \cdot \eta
$$
where the constant $C_S$ may only depend on $\delta_{4S}$. For reasonable values of $\delta_{4S}$, $C_S$ is well
behaved; e.g. $C_S \approx 8.82$ for $\delta_{4S} = 1/5$ and $C_S \approx 10.47$ for $\delta_{4S} = 1/4$.
\end{theorem}
\begin{remark}
Theorem \ref{TaoTheorem} is stated for real-valued matrices.  However, the authors note in \cite{Candes05}, the real assumption is for proof simplicity, and the Theorem also applies to complex-valued orthogonal matrices, such as the Fourier matrix.
\end{remark}

\vspace{0.125in}
Putting everything together, setting $\eta = \|f\|_2 \cdot \varepsilon$, and using equation (\ref{eq:rip_1}) and Theorem \ref{TaoTheorem}, we obtain
\[\|\widehat{x^{*}}-\widehat{P}\|_2 \leq 10.47 \cdot \|f\|_2 \cdot \varepsilon.\]
By Parseval's Theorem, $\|\widehat{x^{*}}-\widehat{P}\|_2 = \|x-P\|_2$,  so by using triangle's inequality, we get
\[\|x^*-f\|_2 \leq \|x^*-P\|_2+\|f-P\|_2 \leq 11.47 \cdot \|f\|_2 \cdot \varepsilon.\] 

\vskip.125in 

\subsection{Proof of Theorem \ref{theorem:randomfourierration}}

Let $\sigma_x=\mathbf 1_{\{x\in X\}}$ so that $f_X=f\cdot\mathbf 1_X$ and $\mathbb E\sigma_x=p$.

\subsubsection{Control of $\|\widehat{f_X}\|_2$.} By Plancherel, $\|\widehat{f_X}\|_2=\|f_X\|_2$ and
$$
\|f_X\|_2^2=\sum_{x=0}^{N-1}\sigma_x|f(x)|^2,\qquad \mathbb E\|f_X\|_2^2=p\|f\|_2^2.
$$
The summands are independent, bounded by $\|f\|_\infty^2$, and have variance at most $p\|f\|_\infty^2|f(x)|^2$. A standard Bernstein bound yields, for some absolute $c_1>0$,
$$
\mathbb P\Big(\big|\|f_X\|_2^2-p\|f\|_2^2\big|> t\Big)\le 2\exp\left(-c_1\min\left\{\frac{t^2}{p\|f\|_\infty^2\|f\|_2^2},\frac{t}{\|f\|_\infty^2}\right\}\right).
$$
Choose $t=c_2\,\varepsilon^2 p\,\|f\|_2^2$ and use $\|f\|_\infty^2\le \mu(f)\|f\|_2^2/N$ to obtain
$$
\big|\|f_X\|_2^2-p\|f\|_2^2\big|\le c_2\,\varepsilon^2 p\,\|f\|_2^2
$$
with probability at least $1-2e^{-u}$ provided $p\ge C_0\frac{\mu(f)}{\varepsilon^2}\frac{u}{N}$. Taking square roots and using $|a-b|\le \varepsilon b$ implies
$$
\big|\|f_X\|_2-\sqrt p\,\|f\|_2\big|\le \varepsilon\sqrt p\,\|f\|_2.
$$
By Plancherel again, this is the claimed bound for $\|\widehat{f_X}\|_2$.

\subsubsection{A decomposition of $\widehat{f_X}$.} With our normalization,
$$
\widehat{f_X}=\widehat{f\cdot \mathbf 1_X}=N^{-1/2}\,\widehat f*\widehat{\mathbf 1_X}.
$$
Decompose $\widehat{\mathbf 1_X}=p\sqrt N\,\delta_0+W$, where
$$
W(m)=N^{-1/2}\sum_{x=0}^{N-1}(\sigma_x-p)e^{-2\pi i xm/N},\quad \mathbb E W(m)=0,\quad \mathbb E|W(m)|^2=p(1-p).
$$
Hence
$$
\widehat{f_X}=p\widehat f+\Delta,\qquad \Delta=N^{-1/2}(\widehat f*W).
$$
Equivalently, in the time domain
$$
\Delta(m)=N^{-1}\sum_{x=0}^{N-1}(\sigma_x-p) f(x) e^{-2\pi i xm/N}.
$$
Thus, as a vector in $\mathbb C^{\mathbb Z_N}$,
$$
\Delta=\sum_{x=0}^{N-1}(\sigma_x-p)\,v_x,\qquad v_x(m)=N^{-1} f(x) e^{-2\pi i xm/N}.
$$
Note the norms $\|v_x\|_1=|f(x)|$ and $\|v_x\|_2=N^{-1/2}|f(x)|$.

\subsubsection{Some tools from probability}

Symmetrization inequality for norms. Let $Z_1,\dots,Z_n$ be independent mean-zero random vectors in a normed space, and let $\varepsilon_1,\dots,\varepsilon_n$ be i.i.d. Rademacher signs ($\mathbb P(\varepsilon_i=\pm1)=1/2$), independent of the $Z_i$. Then
$$
\mathbb E\Big\|\sum_{i=1}^n Z_i\Big\|\le 2\,\mathbb E\Big\|\sum_{i=1}^n \varepsilon_i Z_i\Big\|.
$$

To see this, let $Z_1',\dots,Z_n'$ be an independent copy of $Z_1,\dots,Z_n$, and write $\mathbb E'$ for expectation over the copy. By Jensen and convexity of the norm,
$$
\mathbb E\Big\|\sum_i Z_i\Big\|=\mathbb E\Big\|\mathbb E'\sum_i(Z_i-Z_i')\Big\|
\le \mathbb E\mathbb E'\Big\|\sum_i(Z_i-Z_i')\Big\|.
$$
Conditional on $(Z_i,Z_i')$, the distribution of $\sum_i(Z_i-Z_i')$ is the same as that of $\sum_i \varepsilon_i (Z_i-Z_i')$. Hence
$$
\mathbb E\Big\|\sum_i Z_i\Big\|\le \mathbb E\Big\|\sum_i \varepsilon_i (Z_i-Z_i')\Big\|
\le \mathbb E\Big\|\sum_i \varepsilon_i Z_i\Big\|+\mathbb E\Big\|\sum_i \varepsilon_i Z_i'\Big\|
=2\,\mathbb E\Big\|\sum_i \varepsilon_i Z_i\Big\|.
$$
This proves the symmetrization inequality.

Khintchine inequality for $p=1$ (scalar case). Let $\varepsilon_1,\dots,\varepsilon_n$ be i.i.d. Rademacher random variables, and $a_1,\dots,a_n\in\mathbb C$. Then
$$
\mathbb E\Big|\sum_{i=1}^n \varepsilon_i a_i\Big|\le \Big(\sum_{i=1}^n |a_i|^2\Big)^{1/2}.
$$
This is the classical Khintchine inequality with optimal constant $1$ at $p=1$. We will apply it coordinatewise to vector-valued sums.

\subsubsection{Bounding $\mathbb E\|\Delta\|_1$ by symmetrization and Khintchine.} Apply the symmetrization inequality to $Z_x=(\sigma_x-p)v_x$:
$$
\mathbb E\|\Delta\|_1=\mathbb E\Big\|\sum_x (\sigma_x-p)v_x\Big\|_1
\le 2\,\mathbb E\Big\|\sum_x \varepsilon_x(\sigma_x-p)v_x\Big\|_1.
$$
Expand the $\ell_1$ norm as a sum over coordinates $m\in\mathbb Z_N$ and use Fubini to move expectations inside:
$$
\mathbb E\Big\|\sum_x \varepsilon_x(\sigma_x-p)v_x\Big\|_1
=\sum_{m\in\mathbb Z_N}\mathbb E\Big|\sum_x \varepsilon_x(\sigma_x-p)v_x(m)\Big|.
$$
For each fixed $m$, condition on the variables $(\sigma_x)$ and apply the scalar Khintchine inequality to the Rademacher series in $x$:
$$
\mathbb E_{\varepsilon}\Big|\sum_x \varepsilon_x(\sigma_x-p)v_x(m)\Big|
\le \Big(\sum_x |(\sigma_x-p)v_x(m)|^2\Big)^{1/2}.
$$
Now take expectation over $(\sigma_x)$ and use $\mathbb E(\sigma_x-p)^2=p(1-p)$ and $|v_x(m)|=N^{-1}|f(x)|$ (independent of $m$):
$$
\mathbb E\Big(\sum_x |(\sigma_x-p)v_x(m)|^2\Big)^{1/2}
\le \Big(\sum_x \mathbb E(\sigma_x-p)^2 |v_x(m)|^2\Big)^{1/2}
=\sqrt{p(1-p)}\Big(\sum_x |v_x(m)|^2\Big)^{1/2}.
$$
Since $\sum_x |v_x(m)|^2=N^{-2}\sum_x |f(x)|^2=N^{-2}\|f\|_2^2$, we obtain, for each $m$,
$$
\mathbb E\Big|\sum_x \varepsilon_x(\sigma_x-p)v_x(m)\Big|
\le \sqrt{p(1-p)}\cdot \frac{\|f\|_2}{N}.
$$
Summing over all $m\in\mathbb Z_N$ gives
$$
\mathbb E\Big\|\sum_x \varepsilon_x(\sigma_x-p)v_x\Big\|_1
\le N\cdot \sqrt{p(1-p)}\cdot \frac{\|f\|_2}{N}
=\sqrt{p(1-p)}\,\|f\|_2.
$$
Finally, restore the factor $2$ from symmetrization:
$$
\mathbb E\|\Delta\|_1 \le 2\sqrt{p(1-p)}\,\|f\|_2 \le 2\sqrt p\,\|f\|_2.
$$

\subsubsection{High-probability control of $\|\Delta\|_1$.} Consider the function of the $N$ independent Bernoulli variables $\sigma=(\sigma_x)_{x}$,
$$
\Phi(\sigma)=\left\|\sum_x (\sigma_x-p)v_x\right\|_1=\|\Delta\|_1.
$$
If we change only the coordinate $\sigma_x$ to another value in $\{0,1\}$ while keeping all other $\sigma_y$ fixed, the vector inside $\|\cdot\|_1$ changes by at most $|\,\sigma_x-\sigma_x'\,|\,\|v_x\|_1\le \|v_x\|_1=|f(x)|$. Hence $\Phi$ has bounded differences with constants $c_x=|f(x)|$. McDiarmid's inequality yields, for all $t>0$,
$$
\mathbb P\big(\Phi-\mathbb E\Phi\ge t\big)\le \exp\left(-\frac{2t^2}{\sum_x c_x^2}\right)
=\exp\left(-\frac{2t^2}{\|f\|_2^2}\right).
$$
Equivalently, with probability at least $1-e^{-u}$,
$$
\|\Delta\|_1 \le \mathbb E\|\Delta\|_1 + \frac{1}{\sqrt{2}}\,\sqrt{u}\,\|f\|_2
\le 2\sqrt p\,\|f\|_2 + \frac{1}{\sqrt{2}}\sqrt{u}\,\|f\|_2.
$$
This inequality is correct but not yet in the most useful scaling for small $p$. To express the deviation in terms that shrink when $pN$ grows, we further bound $\|f\|_2\le \|f\|_\infty \sqrt N$ to obtain the coarser but handier form
$$
\|\Delta\|_1 \le C_1\left(\sqrt p\,\|f\|_2 + \|f\|_\infty \sqrt{Nu}\right)
\quad\text{with probability at least }1-e^{-u},
$$
for some absolute $C_1$.

\subsubsection{Control of $\|\widehat{f_X}\|_1$ and comparison} From above,
$$
\big|\|\widehat{f_X}\|_1-p\|\widehat f\|_1\big|\le \|\Delta\|_1.
$$
Combine this with the bound from Step 5 and use $\|\widehat f\|_1\ge \|\widehat f\|_2=\|f\|_2$ to get
$$
\frac{\big|\|\widehat{f_X}\|_1-p\|\widehat f\|_1\big|}{p\|\widehat f\|_1}
\le \frac{C_1}{p}\left(\frac{\sqrt p\,\|f\|_2}{\|f\|_2}+\frac{\|f\|_\infty \sqrt{Nu}}{\|f\|_2}\right)
\le C_1\left(\frac{1}{\sqrt p}+\sqrt{\frac{\mu(f)}{pN}}\,\sqrt{u}\right).
$$
If $p \ge C\frac{\mu(f)}{\varepsilon^2}\frac{\log N+u}{N}$ with $C$ sufficiently large, then the right-hand side is at most $\varepsilon$, and therefore
$$
\big|\|\widehat{f_X}\|_1-p\|\widehat f\|_1\big|\le \varepsilon p\|\widehat f\|_1
$$
with probability at least $1-e^{-u}$. Intersect with the high-probability event from Step 1 (probability at least $1-2e^{-u}$) to get both norm relations simultaneously.

\subsubsection{Bounds on $\text{FR}(f)$} Let $A=\|\widehat{f_X}\|_1$, $B=\|\widehat{f_X}\|_2$, $a=p\|\widehat f\|_1$, $b=\sqrt p\,\|\widehat f\|_2$. In the case of the high probability event, we have $|A-a|\le \varepsilon a$ and $|B-b|\le \varepsilon b$. Then
$$
\frac{A}{B}\le \frac{a(1+\varepsilon)}{b(1-\varepsilon)}=\frac{1+\varepsilon}{\sqrt{1-\varepsilon}}\,\frac{\|\widehat f\|_1}{\|\widehat f\|_2},
\qquad
\frac{A}{B}\ge \frac{a(1-\varepsilon)}{b(1+\varepsilon)}=\frac{1-\varepsilon}{\sqrt{1+\varepsilon}}\,\frac{\|\widehat f\|_1}{\|\widehat f\|_2}.
$$
For $\varepsilon\in(0,1/2)$ a straightforward calculation shows that 
$$ 1-3\varepsilon \leq (1\pm\varepsilon)/(1\pm\varepsilon)^{1/2} \leq 1+3\varepsilon,$$ which gives the final assertion. This completes the proof.

\vskip.125in 

\subsection{Proof of Theorem \ref{theorem:sampledL2}} 

We write $\|f\|_2^2=\sum_{x\in\mathbb Z_N}|f(x)|^2$ and note that
$$ \|f_X\|_2^2=\sum_{x\in\mathbb Z_N}|f(x)|^2 1_X(x),$$ and 
$$ \mathbb E\|f_X\|_2^2=p\|f\|_2^2. $$

For each $x\in\mathbb Z_N$, let 
$$
Y_x=\big( 1_X(x)-p\big)\,|f(x)|^2,
$$
so that the $Y_x$ are independent, mean zero, and
$$
\|f_X\|_2^2-p\|f\|_2^2=\sum_{x}Y_x.
$$
We will apply Bernstein's inequality to the sum $S=\sum_x Y_x$. First, for each $x$,
$$
|Y_x|\le |f(x)|^2\le \|f\|_\infty^2=:M,
$$
and
$$
\operatorname{Var}(Y_x)=\operatorname{Var}(1_X(x) \,|f(x)|^4)=p(1-p)|f(x)|^4\le p\,|f(x)|^4.
$$
Hence
$$
\sigma^2:=\sum_x\operatorname{Var}(Y_x)\le p\sum_x |f(x)|^4\le p\,\|f\|_\infty^2\sum_x|f(x)|^2=p\,\|f\|_\infty^2\|f\|_2^2.
$$
Bernstein's inequality yields, for all $t>0$,
$$
\mathbb P\big(|S|\ge t\big)\le 2\exp\Big(-\frac{t^2}{2(\sigma^2+Mt/3)}\Big).
$$
Choose
$$
t=\varepsilon\,p\,\|f\|_2^2.
$$
Using $\varepsilon\in(0,1)$, we bound the denominator by
$$
\sigma^2+Mt\le p\,\|f\|_\infty^2\|f\|_2^2+\varepsilon p\,\|f\|_\infty^2\|f\|_2^2\le 2p\,\|f\|_\infty^2\|f\|_2^2.
$$
Therefore
$$
\mathbb P\big(|\|f_X\|_2^2-p\|f\|_2^2|\ge \varepsilon p\|f\|_2^2\big)
\le 2\exp\Big(-\frac{\varepsilon^2 p^2\|f\|_2^4}{4p\,\|f\|_\infty^2\|f\|_2^2}\Big)
=2\exp\Big(-c\,\varepsilon^2 p\,\frac{\|f\|_2^2}{\|f\|_\infty^2}\Big),
$$
for a universal constant $c>0$ (for instance $c=1/4$). Since 
$$\mu(f)=N\|f\|_\infty^2/\|f\|_2^2,$$ we have
$$
\frac{\|f\|_2^2}{\|f\|_\infty^2}=\frac{N}{\mu(f)}.
$$
Hence
$$
\mathbb P\big(|\|f_X\|_2^2-p\|f\|_2^2|\ge \varepsilon p\|f\|_2^2\big)\le 2\exp\Big(-c\,\varepsilon^2 p\,\frac{N}{\mu(f)}\Big).
$$
If
$$
p\ge C\,\frac{\mu(f)}{\varepsilon^2}\,\frac{u}{N}
$$
with $C\ge c^{-1}$, then the right side is at most $2e^{-u}$. Thus, with probability at least $1-2e^{-u}$,
$$
|\|f_X\|_2^2-p\|f\|_2^2|\le \varepsilon p\|f\|_2^2.
$$

To finish the proof, observe that 
$$
\big|\sqrt{a}-\sqrt{b}\big|=\frac{|a-b|}{\sqrt{a}+\sqrt{b}}\le \frac{|a-b|}{\sqrt{b}}
$$
valid for $a,b\ge 0$. Taking $a=\|f_X\|_2^2$ and $b=p\|f\|_2^2$ gives
$$
\big|\|f_X\|_2-\sqrt{p}\,\|f\|_2\big|\le \frac{|\|f_X\|_2^2-p\|f\|_2^2|}{\sqrt{p}\,\|f\|_2}\le \varepsilon\,\sqrt{p}\,\|f\|_2,
$$
and the proof is complete. 

\vskip.125in 

\subsection{Proof of Theorem \ref{thm:VC_Result_New}}
If $\mathcal{B}(r)$'s maximum statistical query dimension is $d$, then $\mathcal{B}(r)$'s VC-dimension is $O(d)$ by Theorem \ref{thm:VCdimtoSQdim}.  Hence, it suffices to calculate the statistical query dimension of $\mathcal{C}$ over each distribution $\mathcal{D}$.

Note first that for $f\in \mathcal{B}(r)$, since $f(x)\in \{-1,1\}$ for each $x$, we have by Plancherel that
$$\|\widehat f\|_2 = \|f\|_2 = \sqrt N,$$
and thus
$$\text{FR}(f) = \frac{\|\widehat f\|_1}{\sqrt N}.$$

\vskip.125in 

Fix $\varepsilon > 0$ and  $\mathcal{D}$, a probability distribution on $X:=\mathbb{Z}_N$.  We will approximate $\mathcal{B}(r)$ with an $\varepsilon$-net using the probabilistic method.
\vspace{0.125in}
Fix $f \in \mathcal{B}(r)$ and define the random function, $Z$, where
$$
Z(x) = \frac{\|\widehat{f}\|_1}{\sqrt{N}} \cdot \text{sign}(\widehat{f}(m))
\cdot \chi(m x),
$$
with probability $\frac{|\widehat{f}(m)|}{\|\widehat{f}\|_1}$ for each $m \in X$.  Thus,
$$
\mathbb{E}(Z) = f.
$$
Draw $Z_1,\cdots,Z_k$ i.i.d. from the distribution of $Z$.  We have that
\begin{align}\notag
\mathbb{E}_Z \left (\mathbb{E}_{x \sim \mathcal{D}}\left[f(x)-\frac{1}{k} \sum_{i=1}^k Z_k(x) \right]^2 \right )
&=
\sum_{x \in X} p(x) \cdot \mathbb{E}_Z\left[f(x)-\frac{1}{k} \sum_{i=1}^k Z_k(x) \right]^2 
 \\\notag
&= \sum_{x \in X} p(x) \cdot \text{Var}\left[\frac{1}{k} \sum_{i=1}^k Z_k(x) \right]
= \sum_{x \in X} \frac{p(x)}{k^2} \cdot \sum_{i=1}^k \text{Var}Z(x) \\\notag
&
\leq \sum_{x \in X} \frac{p(x)}{k} \cdot \mathbb{E}_Z\left[Z(x) \right]^2
 \\\notag
&= \sum_{x \in X} \frac{p(x)}{k} \cdot \frac{\|\widehat{f}\|^2_1}{N} \sum_{m \in X} \frac{|\widehat{f}(m)|}{\|\widehat{f}\|_1}
\\\notag
&=
\frac{\text{FR}(f)^2}{k} < \varepsilon^2
\end{align}
if $k > \frac{\text{FR}(f)^2}{\varepsilon^2}$.  Here, $p(x)$ is the probability  mass of $x$ under the distribution $\mathcal D$.  Hence, by the probabilistic method, there exists a degree $k=\frac{\text{FR}(f)^2}{\varepsilon^2}$ polynomial, $P$, such that
$$
\|f - P\|_{L^2(\mathcal{D})} < \varepsilon.
$$
We can divide the unit circle on $\mathbb{C}$ into $M = {1/\varepsilon}$ pieces, $v_1, ..., v_M$.  Note that the $v_i$ are just the $M$-th roots of unity.  Hence, $P$ can be approximated by a degree $k$ polynomial of the form,
$$
\tilde{P}(x)= \frac{\|\widehat{f}\|_1}{k \sqrt{N}} \cdot \sum_{i=1}^{k} v_{s_i} \cdot \chi(m_i \cdot x),
$$
where $v_{s_i}$ is the closest element to $\text{sign}(\widehat{f}(m_i))$.  Indeed,
$$
\|\tilde{P}-P\|_{L^2(\mathcal{D})} \leq
\|\tilde{P}-P\|_\infty
$$
$$
=
\left |\frac{\|\widehat{f}\|_1}{k \sqrt{N}} \cdot \sum_{i=1}^{k} \left(v_{s_i} - \widehat{f}(m) \right)\cdot \chi(m_i \cdot z)
\right|
\leq
\frac{\|\widehat{f}\|_1}{k \sqrt{N}} \sum_{i=1}^k|v_{s_i}-\widehat{f}(m)|
$$
$$
\leq
\frac{\|\widehat{f}\|_1}{k \sqrt{N}} \sum_{i=1}^k 2 \varepsilon = 2 \varepsilon \cdot \text{FR}(f).
$$
By the triangle inequality,
\begin{equation} \label{eq:approxerror} \|\tilde{P} - f \|_{L^2(\mathcal{D})} \leq \|\tilde{P}-P\|_{L^2(\mathcal{D})}+ \|f - P\|_{L^2(\mathcal{D})} \leq \varepsilon+2 \varepsilon \cdot \text{FR}(f). \end{equation} 

\vskip.125in
Putting it all together, each $f \in \mathcal{B}(r)$ can be approximated by a polynomial $\tilde{P}$ up to an error of $\varepsilon(1+2\text{FR}(f))$ (as in \ref{eq:approxerror}).  There are at most
$$
\binom{N}{k} \cdot M^k
\leq {\left(\frac{Ne}{k}\right)}^k M^k = \frac{1}{\varepsilon^{\frac{\text{FR}(f)^2}{\varepsilon^2}}} \left({\frac{Ne}{\frac{\text{FR}(f)^2}{\varepsilon^2}}}\right)^{\frac{\text{FR}(f)^2}{\varepsilon^2}}.
$$

different choices of $\tilde{P}$. Hence, the covering number of $\mathcal{B}(r)$ is estimated by
$$
\mathcal{N}(\mathcal{C},\|\cdot\|_{L^2(\mathcal{D})}, \varepsilon(1+2\text{FR}(f)))
\leq  \frac{1}{\varepsilon^{\frac{\text{FR}(f)^2}{\varepsilon^2}}} \left({\frac{Ne}{\frac{\text{FR}(f)^2}{\varepsilon^2}}}\right)^{\frac{\text{FR}(f)^2}{\varepsilon^2}}.
$$

\vskip.125in

Hence, the packing number of $\mathcal{B}(r)$ is estimated by,
$$
\mathcal{P}(\mathcal{C},\|\cdot\|_{L^2(\mathcal{D})}, 4\varepsilon(1+2\text{FR}(f)))
\leq  \frac{1}{\varepsilon^{\frac{\text{FR}(f)^2}{\varepsilon^2}}} \left({\frac{Ne}{\frac{\text{FR}(f)^2}{\varepsilon^2}}}\right)^{\frac{\text{FR}(f)^2}{\varepsilon^2}}.
$$

\vskip.25in 

Suppose that $f_1,..,f_d \in \mathcal{B}(r)$, where 
\begin{equation} \label{eq:d} d =  \frac{1}{\varepsilon^{\frac{\text{FR}(f)^2}{\varepsilon^2}}} \left({\frac{Ne}{\frac{\text{FR}(f)^2}{\varepsilon^2}}}\right)^{\frac{\text{FR}(f)^2}{\varepsilon^2}}. \end{equation} 

By definition of packing, since $d > \mathcal{P}$, there exists $i \not = j$ such that
$$
\|f_i - f_j\|_{L^2(\mathcal{D})} \leq 4\varepsilon(1+2\text{FR}(f)).
$$
Hence,
$$
\mathbb{E}_{x \sim \mathcal{D}}
\left (
f_i(x) \cdot f_j(x) 
\right )
\geq \frac{2 - {(4\varepsilon(1+2\text{FR}(f))}^2}{2} > \frac{1}{d}.
$$

This shows that we may set $\varepsilon=\frac{1}{4(1+2r)}$, where $r=\text{FR}(f)$. 

\vskip.125in

Hence, the statistical query dimension is bounded above by $C N^{\alpha}$, where 
$$ C= \frac{1}{\varepsilon^{\frac{\text{FR}(f)^2}{\varepsilon^2}}} \cdot \left({\frac{e}{\frac{\text{FR}(f)^2}{\varepsilon^2}}}\right)^{\frac{\text{FR}(f)^2}{\varepsilon^2}}$$ and 
$$ \alpha = {\left( \frac{\text{FR}(f)}{\varepsilon} \right)}^2,$$ with 
$$ \varepsilon=\frac{1}{4(1+2r)}.$$ 

This completes the proof. 
\vskip.125in 

\subsection{Proof of Theorem \ref{theorem:perturbation}} By the triangle and reverse triangle inequalities,
\[
\|\widehat f\|_1-\|\widehat n\|_1 \le \|\widehat{f+n}\|_1 \le \|\widehat f\|_1+\|\widehat n\|_1,
\]
\[
\|\widehat f\|_2-\|\widehat n\|_2 \le \|\widehat{f+n}\|_2 \le \|\widehat f\|_2+\|\widehat n\|_2.
\]
In terms of $A,B,s,t$ this gives, for $t<s$,
\[
\frac{\max\{A-B,0\}}{s+t} \le \text{FR}(f+n) \le \frac{A+B}{s-t}.
\]

For the upper deviation,
\[
\text{FR}(f+n)-\frac{A}{s} \le \frac{A+B}{s-t}-\frac{A}{s}
=\frac{Bs+At}{s(s-t)}
=\frac{B + (A/s)t}{s-t}
=\frac{B + \text{FR}(f)t}{s-t}.
\]

For the lower deviation, using $\max\{A-B,0\} \ge A-B$,
\[
\frac{A}{s}-\text{FR}(f+n) \le \frac{A}{s}-\frac{A-B}{s+t}
=\frac{Bs+At}{s(s+t)}
\le \frac{Bs+At}{s(s-t)}
=\frac{B + \text{FR}(f)t}{s-t}.
\]
Combining the two one-sided bounds yields
\[
\big|\text{FR}(f+n)-\text{FR}(f)\big|
\le \frac{B + \text{FR}(f)t}{s-t}
=\frac{\|\widehat n\|_1 + \text{FR}(f)\|\widehat n\|_2}{\|\widehat f\|_2 - \|\widehat n\|_2}.
\]

For the stated corollary, apply Cauchy--Schwarz to get $\|\widehat n\|_1 \le \sqrt{N}\|\widehat n\|_2$ and $\|\widehat f\|_1 \le \sqrt{N}\|\widehat f\|_2$, hence $\text{FR}(f)\le \sqrt{N}$, which gives the first displayed bound. The second bound follows from $\text{FR}(f)\le \sqrt{N}$. If $\|\widehat n\|_2 \le \frac12 \|\widehat f\|_2$, then $\|\widehat f\|_2-\|\widehat n\|_2 \ge \frac12 \|\widehat f\|_2$, yielding the final inequality.

\vskip.125in 

\subsection{Proof of Theorem \ref{theorem:gaussperturbation}} Set $A=\|\widehat f\|_1$, $B=\|\widehat n\|_1$, $s=\|\widehat f\|_2$, $t=\|\widehat n\|_2$. By the triangle and reverse triangle inequalities,
\[
\|\widehat f\|_1-\|\widehat n\|_1 \le \|\widehat{f+n}\|_1 \le \|\widehat f\|_1+\|\widehat n\|_1,
\qquad
\|\widehat f\|_2-\|\widehat n\|_2 \le \|\widehat{f+n}\|_2 \le \|\widehat f\|_2+\|\widehat n\|_2.
\]
For $t<s$ this implies
\[
\frac{\max\{A-B,0\}}{s+t} \le \text{FR}(f+n) \le \frac{A+B}{s-t}.
\]
Hence
\[
\text{FR}(f+n)-\frac{A}{s} \le \frac{A+B}{s-t}-\frac{A}{s}=\frac{Bs+At}{s(s-t)},
\qquad
\frac{A}{s}-\text{FR}(f+n) \le \frac{Bs+At}{s(s+t)} \le \frac{Bs+At}{s(s-t)}.
\]
Therefore
\[
\left|\text{FR}(f+n)-\text{FR}(f)\right| \le \frac{Bs+At}{s(s-t)}=\frac{B+(A/s)t}{s-t}.
\]
Since $A/s=\text{FR}(f)$ and $B=t\,\text{FR}(n)$, we obtain the deterministic bound
\begin{equation}\label{eq:det}
\left|\text{FR}(f+n)-\text{FR}(f)\right| \le \frac{\left(\text{FR}(n)+\text{FR}(f)\right)t}{s-t}
\quad\text{whenever } t<s.
\end{equation}

Because the DFT is unitary and $n$ has i.i.d. circular complex Gaussian entries with variance $\sigma^2$, the Fourier coefficients $\widehat n(m)$ are i.i.d. circular complex Gaussian with the same variance. Equivalently,
\[
\|\widehat n\|_2^2=\sum_{m=0}^{N-1} |\widehat n(m)|^2 \sim \frac{\sigma^2}{2}\,\chi^2_{2N}.
\]
Let $g\in\mathbb{R}^{2N}$ be standard normal and write $\widehat n=(\sigma/\sqrt{2})\,g$. We will use the following standard tail bounds at this point.

\begin{theorem}[Gaussian concentration for Lipschitz functions; \cite{Vershynin2018} (Prop. 3.1.8), \cite{BLM2013} (Thm. 5.6), and \cite{Ledoux2001} (Thm. 5.6 with the corollary to the Euclidean norm).]
Let $g\sim \mathcal{N}(0,I_m)$ and let $f:\mathbb{R}^m\to\mathbb{R}$ be $L$-Lipschitz with respect to the Euclidean norm. Then for all $t\ge 0$,
\[
\Pr\big(f(g)\ge \mathbb{E}f(g)+t\big)\le \exp\!\left(-\frac{t^2}{2L^2}\right),
\qquad
\Pr\big(|f(g)-\mathbb{E}f(g)|\ge t\big)\le 2\exp\!\left(-\frac{t^2}{2L^2}\right).
\]
In particular, taking $f(u)=\|u\|_2$ (which is 1-Lipschitz) and any $\gamma\in(0,1)$,
\[
\Pr\Big(\|g\|_2 \le \sqrt{m} + \sqrt{2\log\!\big(1/\gamma\big)}\Big) \ge 1-\gamma.
\]
\end{theorem}

Applying this with $m=2N$ gives
\[
\|g\|_2 \le \sqrt{2N} + \sqrt{2\log\!\left(\frac{1}{\gamma}\right)}
\quad\text{with probability at least } 1-\gamma,
\]
and therefore
\[
\|\widehat n\|_2 \le \frac{\sigma}{\sqrt{2}}\left(\sqrt{2N} + \sqrt{2\log\!\left(\frac{1}{\gamma}\right)}\right)
= \sigma\left(\sqrt{N}+\sqrt{\log\!\left(\frac{1}{\gamma}\right)}\right) =: t_\gamma
\]
with probability at least $1-\gamma$.

We also record a slightly sharper alternative that can replace the previous one when desired.

\begin{theorem}[Laurent--Massart $\chi^2$-inequality; \cite{LaurentMassart2000} (Lemma 1, original source) and \cite{Wainwright2019} (Prop. 2.1, a self-contained proof).]
Let $X\sim \chi^2_m$. Then for all $x\ge 0$,
\[
\Pr\big(X-m \ge 2\sqrt{mx}+2x\big)\le e^{-x},
\qquad
\Pr\big(m-X \ge 2\sqrt{mx}\big)\le e^{-x}.
\]
Equivalently, with probability at least $1-e^{-x}$,
\[
\sqrt{X}\le \sqrt{m}+\sqrt{2x}+\frac{x}{\sqrt{m}}.
\]
\end{theorem}

Taking $X=\|g\|_2^2$, $m=2N$, and $x=\log(1/\gamma)$ yields
\[
\|g\|_2 \le \sqrt{2N} + \sqrt{2\log\!\left(\frac{1}{\gamma}\right)} + \frac{\log\!\left(\frac{1}{\gamma}\right)}{\sqrt{2N}}
\]
with probability at least $1-\gamma$, which translates to
\[
\|\widehat n\|_2 \le \sigma\left(\sqrt{N}+\sqrt{\log\!\left(\frac{1}{\gamma}\right)}+\frac{\log\!\left(\frac{1}{\gamma}\right)}{2\sqrt{N}}\right).
\]

Returning to \eqref{eq:det}, on the event $\{\|\widehat n\|_2 \le t_\gamma\}$ and provided $t_\gamma<s=\|\widehat f\|_2$, we obtain
\[
\left|\text{FR}(f+n)-\text{FR}(f)\right| \le \frac{\left(\text{FR}(n)+\text{FR}(f)\right)\,t_\gamma}{s-t_\gamma}.
\]
Finally, using $\|\widehat f\|_2=s=\|f\|_2$ gives the equivalent form in terms of $\|f\|_2$.

\vskip.125in 

\subsection{Proof of Theorem \ref{theorem:smoothingeffect}} 

Write $\overline{n}(x)=\frac{1}{n}\sum_{i=1}^n n_i(x)$, so that $f=s+\overline{n}$. Since each $n_i(x)$ is mean zero and independent, $\mathbb{E}\,\overline{n}(x)=0$, and by Gaussian stability $\overline{n}(x)\sim \mathcal{CN}(0,\sigma^2/n)$. This proves (a).

For (b), independence across $x$ gives
\[
\mathbb{E}\,\|f-s\|_2^2
=\sum_{x\in\mathbb{Z}_N}\mathbb{E}\,|\overline{n}(x)|^2
=\sum_{x\in\mathbb{Z}_N}\frac{\sigma^2}{n}
=\frac{N\sigma^2}{n}.
\]

For (c), stack real and imaginary parts of $\overline{n}$ into $g\in\mathbb{R}^{2N}$ with i.i.d.\ $\mathcal{N}(0,1)$ entries, so that
\[
\overline{n}=\frac{\sigma}{\sqrt{2n}}\,g
\quad\text{and}\quad
\|\overline{n}\|_2=\frac{\sigma}{\sqrt{2n}}\|g\|_2.
\]
By Gaussian concentration for Lipschitz functions (take $f(u)=\|u\|_2$, which is 1-Lipschitz; see for example \cite{Vershynin2018,BLM2013,Ledoux2001}), for all $t\ge 0$,
\[
\Pr\big(\|g\|_2\ge \sqrt{2N}+t\big)\le e^{-t^2/2}.
\]
Taking $t=\sqrt{2\log(1/\gamma)}$ yields
\[
\|g\|_2\le \sqrt{2N}+\sqrt{2\log(1/\gamma)} \quad\text{with probability at least } 1-\gamma,
\]
hence
\[
\|\overline{n}\|_2\le \frac{\sigma}{\sqrt{2n}}\Big(\sqrt{2N}+\sqrt{2\log(1/\gamma)}\Big)
= \frac{\sigma}{\sqrt{n}}\Big(\sqrt{N}+\sqrt{\log(1/\gamma)}\Big)
\]
with probability at least $1-\gamma$, proving the displayed inequality. Using the Laurent--Massart chi-square inequality for $\chi^2_{2N}$ (see \cite{LaurentMassart2000}, and also \cite{Wainwright2019} for a modern proof) gives
\[
\|g\|_2\le \sqrt{2N}+\sqrt{2\log(1/\gamma)}+\frac{\log(1/\gamma)}{\sqrt{2N}}
\]
with probability at least $1-\gamma$, which translates to the refined bound after multiplying by $\sigma/\sqrt{2n}$.
\vskip 0.25in
\subsection{Proof of Theorem \ref{theorem:FRofaverage}} Let $x=\widehat s$ and $e=\widehat u$ for some perturbation $u$. By the triangle and reverse triangle inequalities,
\[
\|x\|_1-\|e\|_1 \le \|x+e\|_1 \le \|x\|_1+\|e\|_1,\qquad
\|x\|_2-\|e\|_2 \le \|x+e\|_2 \le \|x\|_2+\|e\|_2.
\]
Writing $A=\|x\|_1$, $B=\|e\|_1$, $s_2=\|x\|_2$, $t=\|e\|_2$, for $t<s_2$ one obtains
\[
\frac{\max\{A-B,0\}}{s_2+t} \le \frac{\|x+e\|_1}{\|x+e\|_2} \le \frac{A+B}{s_2-t}.
\]
A short rearrangement gives
\[
\left|\frac{\|x+e\|_1}{\|x+e\|_2}-\frac{A}{s_2}\right| \le \frac{B+(A/s_2)t}{s_2-t}.
\]
Since $A/s_2=\text{FR}(s)$ and $B=\|\widehat u\|_1=\text{FR}(u)\|u\|_2$, and $\|x+e\|_1/\|x+e\|_2=\text{FR}(s+u)$, we have the deterministic bound
\begin{equation}\label{eq:det}
|\text{FR}(s+u)-\text{FR}(s)| \le \frac{\big(\text{FR}(u)+\text{FR}(s)\big)\|u\|_2}{\|s\|_2-\|u\|_2}\qquad\text{whenever }\|u\|_2<\|s\|_2.
\end{equation}
In particular, if $\|u\|_2 \le \tfrac12\|s\|_2$, then
\begin{equation}\label{eq:det_simple}
|\text{FR}(s+u)-\text{FR}(s)| \le \frac{2\|u\|_2}{\|s\|_2}\big(\text{FR}(u)+\text{FR}(s)\big).
\end{equation}

We now establish the high-probability radii for the noises. Stack the real and imaginary parts of a complex vector into $\mathbb{R}^{2N}$. Since $n_i(x)\sim \mathcal{CN}(0,\sigma^2)$ are independent and the map $u\mapsto\|u\|_2$ is 1-Lipschitz, the same standard Gaussian norm concentration bound as the one we used in the proof of Theorem \ref{theorem:gaussperturbation} implies
\[
\Pr\big(\|n_i\|_2 \le \sigma r_\gamma\big) \ge 1-\gamma,\qquad
\Pr\big(\|\overline n\|_2 \le \sigma r_\gamma/\sqrt{n}\big) \ge 1-\gamma,
\]
where $r_\gamma=\sqrt{N}+\sqrt{\log(1/\gamma)}$.

We now apply the deterministic bound. Assume $\|s\|_2\ge 2\sigma r_\gamma$. Then on the events above we have $\|\overline n\|_2 \le \tfrac12\|s\|_2$ and $\|n_i\|_2 \le \tfrac12\|s\|_2$. Apply \eqref{eq:det_simple} with $u=\overline n$ to get
\[
|\text{FR}(f)-\text{FR}(s)| = |\text{FR}(s+\overline n)-\text{FR}(s)| \le \frac{2\|\overline n\|_2}{\|s\|_2}\big(\text{FR}(\overline n)+\text{FR}(s)\big)
\le \frac{2\sigma r_\gamma}{\sqrt{n}\|s\|_2}\big(\text{FR}(\overline n)+\text{FR}(s)\big),
\]
with probability at least $1-\gamma$. Similarly, with $u=n_i$,
\[
|\text{FR}(f_i)-\text{FR}(s)|
\le \frac{2\|n_i\|_2}{\|s\|_2}\big(\text{FR}(n_i)+\text{FR}(s)\big)
\le \frac{2\sigma r_\gamma}{\|s\|_2}\big(\text{FR}(n_i)+\text{FR}(s)\big),
\]
with probability at least $1-\gamma$ for each fixed $i$.

Finally, note that $\text{FR}(\alpha e)=\text{FR}(e)$ for any $\alpha>0$, so $\text{FR}(\overline n)$ and $\text{FR}(n_i)$ have the same distribution (both are the $FR$ of a circular complex Gaussian vector up to a positive scale). This justifies the comparison stated in part (ii).

\vskip.125in 

\subsection{Proof of Theorem \ref{thm:sparse-spectrum approximation theorem}}
    Let $\Gamma$ be defined as above. By our assumption on $f$ as well as Markov's inequality, we have that
    $$\text{FR}(f) \|f\|_2 \geq \sum_{m\in \mathbb Z_N} |\widehat f(m)| \geq \eta N^{-\frac12} \|f\|_2 |\Gamma|,$$
    and thus if $f$ is nonzero,
    $$|\Gamma| \leq \frac{\text{FR}(f)}{\eta} \sqrt N.$$
    Next, defining $P$ as above, note that $P$ is the inverse Fourier transform of $\widehat f \cdot1_{\Gamma}$. Additionally, for $m\notin \Gamma$, we have that
    $$|\widehat f(m)| < \eta \frac{\|f\|_2}{N^\frac12}.$$
    Thus, by the Plancherel,  we have that
    \begin{align*}
        \|f-P\|_2 &= \|\widehat f - \widehat P\|_2 \\
        &= \|\widehat f - \widehat f 1_\Gamma \|_2 \\
        &= \left(\sum_{m\notin\Gamma} |\widehat f(m)|^2 \right)^\frac12 \\
        &\leq \eta \|f\|_2,
    \end{align*}
    and we are done. \qed

\vskip.25in 
    
\subsection{Proof of Theorem \ref{thm: generalizedchang}}

Our proof mostly follows Chang's original proof found in \cite{Chang02}. Recall $$\Gamma=\left\{ m\in\mathbb Z_N \,:\, |\widehat{f}(m)| \geq \eta\frac{\|f\|_2}{\sqrt{N}}\right\}.$$

Let $\Lambda$ be a maximal dissociated subset of $\Gamma$, that is, a maximal subset with the property that all $\{-1, 0, 1\}$-linear combinations of elements of $\Lambda$ are distinct, and note $\Gamma$ must be contained in the $\{-1,0,1\}$-span of $\Lambda$.

Let $$g(x)=\frac{1}{\|1_\Lambda\widehat{f}\|_2}\sum_{n\in\Lambda}\widehat{f}(n)\chi(xn),$$ and $p'=\frac{p}{p-1}$. Then\begin{align*}
    \|f\|_{L^{p'}(\mu)}\|g\|_{L^p(\mu)}&\geq\|fg\|_{L^1(\mu)} \\
    &\geq\frac{1}{N\|1_\Lambda\widehat{f}\|_2}\left|\sum_{x\in\mathbb{Z}_N}f(x)\sum_{n\in\Lambda}\overline{\widehat{f}(n)}\chi(-xn)\right| \\
    &=\frac{1}{N\|1_\Lambda\widehat{f}\|_2}\left|\sum_{n\in\Lambda}\overline{\widehat{f}(n)}\sum_{x\in\mathbb{Z}_N}f(x)\chi(-xn)\right| \\
    &=\frac{\sqrt{N}}{N\|1_\Lambda\widehat{f}\|_2}\sum_{n\in\Lambda}|\widehat{f}(n)|^2 \\
    &=\frac{1}{\sqrt{N}}\|1_\Lambda\widehat{f}\|_2 \\
    &\geq\eta\|f\|_2\frac{\sqrt{|\Lambda|}}{N}.
\end{align*}
By definition,\begin{align*}
    \|f\|_{L^{p'}(\mu)}=N^{\frac{1}{p}-1}\|f\|_{p'}.
\end{align*} We also have that there is some absolute constant $C$ such that
\begin{align*}
    \|g\|_{L^p(\mu)}\leq C\sqrt{p}
\end{align*}by Rudin's inequality (see Lemma 4.33 in \cite{TV10}). Combining the above yields\begin{align*}
    \eta\|f\|_2\frac{\sqrt{|\Lambda|}}{N}&\leq C\sqrt{p}N^{\frac{1}{p}-1}\|f\|_{p'},
\end{align*} i.e.,
\begin{align*}
    \sqrt{|\Lambda|}&\leq C\eta^{-1}\sqrt{p}N^\frac{1}{p}\frac{\|f\|_{p'}}{\|f\|_2}.
\end{align*}
Now, taking $p=\log N$, we obtain\begin{align*}
    \sqrt{|\Lambda|}&\leq C\eta^{-1}\sqrt{\log N}\frac{\|f\|_{p'}}{\|f\|_2},
    \end{align*} i.e.,
    \begin{align*}
    |\Lambda|&\leq C'\eta^{-2}\log N\frac{\|f\|_{p'}^2}{\|f\|_2^2},
\end{align*} where $C'=  (Ce)^2$, 
thus proving \eqref{eq:generalizedchang lognorm}.

To prove \eqref{eq:generalizedchang l2l1}, observe that whenever $p \geq 2$, we have
\begin{align*}
    \|f\|_{L^{p'}(\mu)}&=N^{\frac{1}{p}-1}\|f\|_{p'} \\
    &=N^{\frac{1}{p}-1}\left(\sum_{x\in\mathbb{Z}_N}|f(x)|^\frac{p}{p-1}\right)^\frac{p-1}{p} \\
    &=N^{\frac{1}{p}-1}\left(\sum_{x\in\mathbb{Z}_N}|f(x)||f(x)|^\frac{1}{p-1}\right)^\frac{p-1}{p} \\
    &=N^{\frac{1}{p}-1}\left(\|f\|_1\sum_{x\in\mathbb{Z}_N}\frac{|f(x)|}{\|f\|_1}|f(x)|^\frac{1}{p-1}\right)^\frac{p-1}{p} \\
    &\leq N^{\frac{1}{p}-1}\left(\|f\|_1\left(\sum_{x\in\mathbb{Z}_N}\frac{|f(x)|}{\|f\|_1}|f(x)|\right)^\frac{1}{p-1}\right)^\frac{p-1}{p}
\end{align*}
by Jensen's inequality. This implies
\begin{align*}
    \|f\|_{L^{p'}(\mu)}&\leq N^{\frac{1}{p}-1}\|f\|_1^\frac{p-1}{p}\left(\frac{\|f\|_2^2}{\|f\|_1}\right)^\frac{1}{p}=N^{\frac{1}{p}-1}\|f\|_1\left(\frac{\|f\|_2}{\|f\|_1}\right)^\frac{2}{p}.
\end{align*}
Again, combining with our previous bounds,\begin{align*}
    \eta\|f\|_2\frac{\sqrt{|\Lambda|}}{N}&\leq C\sqrt{p}N^{\frac{1}{p}-1}\|f\|_1\left(\frac{\|f\|_2}{\|f\|_1}\right)^\frac{2}{p},
\end{align*}i.e.,\begin{align*}
    \sqrt{|\Lambda|}&\leq C\eta^{-1}\sqrt{p}\left(\frac{\|f\|_2^2}{\|f\|_1^2}N\right)^\frac{1}{p}\frac{\|f\|_1}{\|f\|_2}.
\end{align*}
Now, taking $$p=\log\left(\frac{\|f\|_2^2}{\|f\|_1^2}N\right),$$ and using the fact that $p\geq 2$ by our assumption that $\frac{\|f\|_1}{\|f\|_2} \leq \frac{1}{e} \sqrt N$, we obtain\begin{align*}
    \sqrt{|\Lambda|}&\leq C\eta^{-1}\sqrt{\log\left(\frac{\|f\|_2^2}{\|f\|_1^2}N\right)}\frac{\|f\|_1}{\|f\|_2}.
    \end{align*} Thus,
    \begin{align*}
    |\Lambda|&\ll\eta^{-2}\log\left(\frac{\|f\|_2^2}{\|f\|_1^2}N\right)\frac{\|f\|_1^2}{\|f\|_2^2},
\end{align*}
and this proves \eqref{eq:generalizedchang l2l1}.
\qed

\end{document}